\documentclass[english]{amsart}
\usepackage{babel}
\usepackage[margin=1.5in]{geometry}

\usepackage{caption}

\usepackage{pictexwd}
\usepackage{epsfig}
\usepackage{pst-grad} 
\usepackage{pst-plot} 

\newcommand{\R}{\mathbb{R}}

\newcommand{\eps}{\varepsilon}

\theoremstyle{plain}
\newtheorem{defi}{Definition}[section]
\newtheorem{proposition}[defi]{Proposition}
\newtheorem{theorem}[defi]{Theorem}

\newtheorem{lemma}[defi]{Lemma}

\newtheorem{remark}[defi]{Remark}
\newtheorem{notation}[defi]{Notation}

\theoremstyle{definition}

\theoremstyle{remark}

\numberwithin{equation}{section}

\usepackage{tikz}

\begin{document}

\title[Single-point gradient blow-up]{Single-point Gradient Blow-up on the Boundary for Diffusive Hamilton-Jacobi Equation in domains with non-constant curvature.}

\author{Carlos Esteve}
\address{Universit\'{e} Paris 13, Sorbonne Paris Cit\'{e}, Laboratoire Analyse, G\'{e}om\'{e}trie et Applications, 93430, Villetaneuse, France.
\newline {\tt E-mail address: esteve@math.univ-paris13.fr}
}

\date{\today}

\begin{abstract}
We consider  the diffusive Hamilton-Jacobi equation 
$u_t - \Delta u = |\nabla u|^p$ in a bounded planar domain with zero Dirichlet boundary condition.
It is known that, for $p>2$, the solutions to this problem can exhibit gradient blow-up (GBU) at the boundary. 
In this paper we study the possibility of the GBU set being reduced to a single point.
In a previous work \textit{[Y.-X. Li, Ph. Souplet, 2009]}, it was shown that single point GBU solutions can be constructed in 
very particular domains, i.e.~locally flat domains and disks. Here, we prove the existence of single point GBU solutions in a large class of domains, 
for which the curvature of the boundary may be nonconstant near the GBU point.

Our strategy is to use a boundary-fitted curvilinear coordinate system, combined with suitable auxiliary functions
and appropriate monotonicity properties of the solution.
The derivation and analysis of the parabolic equations satisfied by the auxiliary functions necessitate
long and technical calculations involving boundary-fitted coordinates.
\end{abstract}

\maketitle

\tableofcontents

\section{Introduction and first results}\label{Introduction}

We consider the initial-boundary value problem for the diffusive Hamilton-Jacobi equation
\begin{equation}\label{Hamilton-Jacobi problem}
\left\lbrace \begin{array}{rrlll}
u_t - \Delta u & =& |\nabla u|^p, & x\in \Omega, & t>0, \\
 \noalign{\vskip 1mm}
u &=& 0, & x\in \partial \Omega, & t>0, \\
 \noalign{\vskip 1mm}
u( x,0) &=& u_0 (x), & x \in \Omega, &
\end{array} \right.
\end{equation}
where $\Omega$ is a smooth bounded domain in $\mathbb{R}^2$, $p>2$ and 
\begin{equation*}
u_0 \in X_+ := \{ v\in C^1(\overline{\Omega}); \ v\geq 0, \ v|_{\partial\Omega} = 0 \}.
\end{equation*}

Equation \eqref{Hamilton-Jacobi problem} is a typical model-case in the theory of nonlinear parabolic equations,
being the simplest example of a parabolic equation with a nonlinearity depending on the gradient of the solution.
It has been extensively studied in the past twenty years and it is well known that
 if $p\leq 2$ or if $\Omega=\R^n$, then 
all solutions exist globally in the classical sense, see
\cite{ABA}, \cite{BKL}, \cite{BL}, \cite{BASW}, \cite{G}, \cite{GGK}, \cite{LS}, \cite{QS}, \cite{S2}.
On the contrary,
for the case of superquadratic growth of the nonlinearity, i.e. $p>2$, with $\Omega\neq\R^n$, 
solutions exhibit singularities for large enough initial data.
The nature of this singularity is of \emph{gradient blow-up} type, and occurs on some subset of the boundary of the domain, see
\cite{A}, \cite{ABG}, \cite{ARBS}, \cite{BDL}, \cite{CG}, \cite{FL}, \cite{GH}, \cite{HM}, \cite{Li-Souplet},
\cite{QS}, \cite{S1}, \cite{SV}, \cite{Souplet-Zhang}.

In addition, equation \eqref{Hamilton-Jacobi problem} arises in
stochastic control theory \cite{Lions}, and is involved in certain physical models, for example of ballistic deposition processes,
where the solution describes the growth of an interface, see \cite{HH-Z}, \cite{KPZ}, \cite{KS}.

It follows from classical theory, see for example \cite[Theorem 10, p. 206]{Friedman}, that problem \eqref{Hamilton-Jacobi problem} admits a unique maximal, 
nonnegative classical solution $u\in C^{2,1}(\overline{\Omega}\times (0,T))\cap C^{1,0} (\overline{\Omega}\times[0,T))$,
where $T=T(u_0)$ is the maximal existence time.
By the maximum principle, for problem \eqref{Hamilton-Jacobi problem} we have
$$
\| u(t)\|_\infty \leq \|u_0\|_\infty, \qquad 0<t<T.
$$
Since \eqref{Hamilton-Jacobi problem} is well posed in $X_+$, it follows that, if $T<\infty$, then
$$
\lim_{t\to T} \| \nabla u(t) \|_\infty = \infty.
$$
This phenomenon of $\nabla u$ blowing up with $u$ remaining uniformly bounded is known as \emph{gradient blow-up}. 
The gradient blow-up set of $u$ is defined by
$$GBUS(u_0) =   \{x_0\in \partial\Omega; \ \displaystyle\limsup_{t\to T,\,x\to x_0} |\nabla u(x,t)|=\infty\}.$$
We call
gradient blow-up point (GBU point for short) 
any point in $GBUS(u_0)$.
The space profile at $t=T$ is investigated in \cite{CG}, \cite{ARBS}, \cite{Souplet-Zhang}, \cite{GH}, \cite{PS}.
For results on the GBU rate, we refer to \cite{CG}, \cite{GH}, \cite{ZL}, \cite{PS2}.
Also, the existence and properties of a weak continuation of the solution after GBU are studied in 
 \cite{FL}, \cite{BDL}, \cite{PZ}, \cite{PS2a}, \cite{QR18}.

From \cite[Theorem 3.2]{Souplet-Zhang}, it follows that gradient blow-up can only occur at the boundary (see also \cite{ABG}, \cite{ARBS}). 
More precisely,
 the following estimate  is given:
\begin{equation}\label{bernstein estimate 1}
|\nabla u|\leq C_1 \delta^{-\frac{1}{p-1}}(x,y) + C_2 \quad \text{in} \ \Omega\times[0,T),
\end{equation}
where $C_1 = C_1(n,p)>0$ and $C_2 = C_2(p,\Omega,\|u_0\|_{C^1})>0$. 
Here, $\delta(x,y)$ is the distance function to the boundary.

In this paper we are interested in the possibility of having isolated gradient blow-up points at the boundary. 
Up to now, the only available results of this kind, ensuring single-point GBU for suitable initial data, are those from \cite{Li-Souplet},
and they are restricted to very particular domains, namely disks and locally flat domains with some symmetry assumptions
(see also \cite{AS17} for a related problem with nonlinear diffusion in locally flat domains).

As it turns out, a key feature in the proofs in
 \cite{Li-Souplet}, \cite{AS17} is the fact that the curvature of the boundary is constant near the GBU point. 
In this paper we are able to show that this can be considerably relaxed and we cover large classes of domains.

In order to give a good illustration of our main results without entering into too much technicality, let us right away
formulate a single point gradient blow-up result for two typical classes of domains.
More general results will be given in Section~2.
We first treat the case of ellipses.

\begin{theorem}\label{ellipses theorem}
Let $p>2$ and $\Omega\subset\mathbb{R}^2$ be an ellipse.
Then, there exist initial data $u_0\in X_+$ such that $T(u_0)<\infty$ and $GBUS(u_0)$ contains only
a boundary point of minimal curvature. 
\end{theorem}

For our second class of domains, the main feature is that the GBU point has its center of curvature lying outside $\Omega$
and is a local minimum of the curvature,
along with suitable geometric conditions. Namely, we assume:
\begin{eqnarray}
\label{x symmetry and convexity pre}
&&\begin{array}{l}
\text{$\Omega$ is symmetric with respect to the line $x=0$ and convex in the $x$-direction},
\end{array}\\
\label{Omega tangent to y>0}
&&\begin{array}{l} \text{$\partial\Omega$ is tangent to the line $y=0$ at the origin and $\Omega\subset\{y>0\}$}, 
\end{array} \\
\label{small curvature at the origin pre}
&&\begin{array}{l}
\text{The radius of curvature $R(x)$ of $\partial \Omega$ is a nonincreasing function for $x>0$ small} \\ 
\text{and $\overline{\Omega}\subset \{y<R (0)\}$},
\end{array} \\
\label{symmetry K(0) pre}
&&\begin{array}{l}
\text{For all $X_0\in \partial\Omega\cap \{x>0\}$ close to the origin, the symmetric of $\Omega_{X_0}$ with respect} \\
\text{to $\Lambda_{X_0}$ is contained in $\Omega$, where $\Lambda_{X_0}$ is the normal line to $\partial\Omega$ at $X_0$, and $\Omega_{X_0}$ is} \\
\text{the part of $\Omega$ to the right of $\Lambda_{X_0}$.}
\end{array}
\end{eqnarray}

See Figure \ref{figure second class} for an example of a domain satisfying these hypotheses.
We point out that the function $R(x)$ in \eqref{small curvature at the origin pre} is valued in $(0,\infty]$.

\begin{theorem}\label{almost flat theorem}
Let $p>2$ and suppose $\Omega\subset\mathbb{R}^2$ is a domain satisfying 
\eqref{x symmetry and convexity pre}--\eqref{symmetry K(0) pre}.
Then, there exist initial data $u_0\in X_+$ such that $T(u_0)<\infty$ and $GBUS(u_0)$ contains only the origin.
\end{theorem}

\begin{remark}
\begin{enumerate} 
\item Observe that in the case of the locally flat domains studied in \cite{Li-Souplet}, 
condition~\eqref{symmetry K(0) pre} is a consequence of \eqref{x symmetry and convexity pre}. In this case, for any $X_0\in\partial\Omega\cap \{x>0\}$ near the origin, $\Lambda_{X_0}$ will
be parallel to the line $x=0$. 
Also hypothesis \eqref{small curvature at the origin pre} is trivially satisfied by locally flat domains.
\item Although it is possible to construct initial data for which the GBU set is 
arbitrarily concentrated close to any given boundary point
(see Proposition \ref{GBU location lemma}),
it is presently a (probably difficult) open question whether single point GBU may occur on points other than local minima of the curvature.
\end{enumerate}
\end{remark}

\begin{figure}[h]
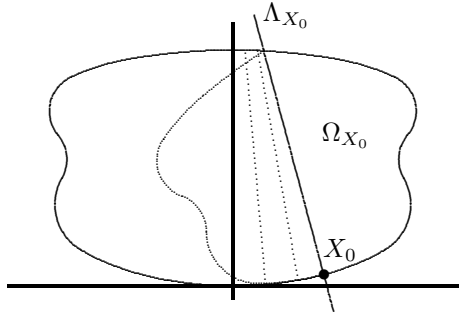

\[
\beginpicture
\setcoordinatesystem units <1cm,1cm>
\setplotarea x from -3 to 3, y from -0.2 to 3.5

\setdots <0pt>
\linethickness=1pt
\putrule from -3 0 to 3 0
\putrule from 0 -0.2 to 0 3.5

\circulararc 10 degrees from 0 0  center at 0 5
\circulararc 10 degrees from 0.86824 0.07596  center at 0.17365 4.01519
\circulararc 10 degrees from 1.54173 0.25642  center at 0.51567 3.07550
\circulararc 90 degrees from 2.01567 0.47742  center at 1.66567 1.08364
\circulararc -70 degrees from 2.27189 1.43364  center at 2.70490 1.68364
\circulararc 105 degrees from 2.32188 2.00504  center at 1.93885 2.32643
\circulararc 15 degrees from 2.15016 2.77958  center at 0.88231 0.06066
\circulararc 10 degrees from 1.40325 3.01508  center at 0.0000 -4.94275

\circulararc -10 degrees from 0 0  center at 0 5
\circulararc -10 degrees from -0.86824 0.07596  center at -0.17365 4.01519
\circulararc -10 degrees from -1.54173 0.25642  center at -0.51567 3.07550
\circulararc -90 degrees from -2.01567 0.47742  center at -1.66567 1.08364
\circulararc 70 degrees from -2.27189 1.43364  center at -2.70490 1.68364
\circulararc -105 degrees from -2.32188 2.00504  center at -1.93885 2.32643
\circulararc -15 degrees from -2.15016 2.77958  center at -0.88231 0.06066
\circulararc -10 degrees from -1.40325 3.01508  center at 0.0000 -4.94275

\setdots <1pt>
\circulararc -5 degrees from 1.208900 0.151490  center at 0.173620 4.015210
\circulararc -10 degrees from 0.868215 0.075962  center at 0.347269 3.030398
\circulararc -90 degrees from 0.347268 0.030385  center at 0.347268 0.730388
\circulararc 70 degrees from -0.352735 0.730388  center at -0.852737 0.730388
\circulararc -105 degrees from -0.681726 1.200237  center at -0.510715 1.670085
\circulararc -15 degrees from -0.920293 1.956874  center at 1.537174 0.236137
\circulararc -7 degrees from -0.391197 2.534281  center at 4.802934 -3.655845

\setdots <0pt>
\setlinear 
\plot 1.33833 -0.33147  0.28215 3.61028 /

\setdots <2pt>
\setlinear
\plot 0.86824 0.075961  0.32955 3.1310 /
\plot 0.43578 0.019027  0.16307 3.1361 /

\put{$X_0$}[ld] at 1.2 0.48
\put{$\bullet$}[cc] at 1.2089 0.15149
\put{$\Lambda_{X_0}$}[lc] at 0.4 3.6
\put{$\Omega_{X_0}$}[cc] at 1.5 2

\endpicture
\]
\caption{Example of domain satisfying hypotheses \eqref{x symmetry and convexity pre}--\eqref{symmetry K(0) pre}.}
\label{figure second class}
\end{figure}

In the next section we give single point GBU results more general than Theorems \ref{ellipses theorem} and \ref{almost flat theorem}, at the expense of more technical statements
(see Theorems \ref{local result} and \ref{second class theorem}).
The technical complexity of the statements comes from the fact that, in order to describe the hypotheses involved, 
we need to introduce a coordinate system adapted to the boundary near the gradient blow-up point
(and actually this coordinate system is crucially used in the proof of our results).

\section{General results}\label{general results} 

We introduce a class of symmetric domains with respect to the line $x=0$,
containing those described in the previous theorems,
and for which we can construct single-point GBU solutions. 
A first step of our strategy is to prove that the solution $u$ is monotone in the parallel direction to the boundary in a neighborhood of the GBU point.
It is therefore natural to introduce a curvilinear coordinate system adapted to the domain,
allowing us to study the sign of the derivative of the solution in the parallel direction to the boundary.
This coordinate system is sometimes called ``boundary-fitted'' coordinate system or ``flow coordinates''.
We point out that the use of these coordinates brings some technical difficulties, and
that long computations and quite delicate arguments are required in order to control the terms related to the non-constant curvature
(under appropriate assumptions on the domain).
However, our attempts to prove such results, on single-point GBU in domains with nonconstant curvature,
 by merely using cartesian coordinates or local charts have turned out to be unsuccessful.

Next, we set the notation used throughout the rest of the paper and
introduce the curvilinear coordinate system mentioned above.
See Figure \ref{Fig Notation 1} for an illustration of this notation.

\begin{notation}\label{notation1}
$\phantom{}$

$\bullet$ $\Omega$ is a smoothly bounded domain of $\mathbb{R}^2$ and 
$\nu=(\nu_x,\nu_y)$ denotes the unit normal outward vector to $\partial \Omega$.

$\bullet$ $\Gamma\subset \partial \Omega$ is a connected boundary piece, with $(0,0)\in\Gamma$,
and we assume that
\begin{equation}\label{symmetry of the domain}
\text{$\Omega$ and $\Gamma$ are symmetric with respect to the line $x=0$.}
\end{equation} 

$\bullet$ For given $s_0>0$, the map 
$$\gamma(s)=(\alpha(s), \beta(s)), \quad\hbox{ $s\in [-s_0,s_0]$,}$$
 is an arclength parametrization of $\Gamma$ 
(i.e.~$\alpha'(s)^2+\beta'(s)^2=1$), with
$\gamma(0) = (0,0)$.

$\bullet$ We denote
$$T(s)=(\alpha'(s),\beta'(s)), \ \quad N(s)=T^\perp(s)=(-\beta'(s),\alpha'(s)),\qquad\hbox{for all $s\in [-s_0,s_0]$.}$$
We see that $T(s)$ is a unit tangent vector to $\partial\Omega$ at the point $\gamma(s)$ and, without loss of generality
(replacing $s$ by $-s$ if necessary), we can assume that 
\begin{equation}\label{good sense of the parametrization}
N(s) \ \text{is the inward normal vector to $\partial\Omega$ at the point $\gamma(s)$}
\end{equation}
and that
$$\gamma(0)=(0,0),\ \quad T(0)=(1,0), \ \quad N(0)=(0,1).$$

$\bullet$ 
We denote the curvature of the boundary by
$$ 
K(s):= det(\gamma',\gamma'')=\alpha'\beta'' - \beta'\alpha'', \quad\hbox{for all $s\in [-s_0,s_0]$.}
$$
By the regularity of $\partial\Omega$, this function is bounded and smooth. 

$\bullet$
We introduce the map $M:=\gamma+rN$, i.e.
\begin{equation}\label{flow coordinates}
\begin{array}{cccc}
M: & [0,\infty)\times [-s_0,s_0] & \longrightarrow & \mathbb{R}^2 \\
       &  (r,s)                  & \longmapsto     & M(r,s)=\gamma(s)+rN(s).
\end{array}
\end{equation}
\end{notation}

For a given domain $\Omega$ and a boundary piece $\Gamma$ as in Notation \ref{notation1}, 
our goal will be to prove the existence of initial data for which the GBU set is reduced to the origin.
Using the coordinates given by the map $M$, 
we will use auxiliary functions to estimate the derivative of $u$ with respect to $s$.
Then, an integration over the coordinate curves parallel to the boundary will give 
an upper estimate on $u$ which is sufficient to apply a nondegeneracy result (see Lemma \ref{nondegeneracy lemma} below) for each $s>0$,
proving that gradient blow-up can only take place at the origin.

In order to apply our methods, we need to make some extra geometric assumptions on the domain. 
Namely, we need to assume that $\Omega$ is locally convex near the origin 
and that the origin is a local minimum for the
curvature of the boundary, i.e.
\begin{equation}\label{increasing curvature}
K(0) \geq 0 \quad \text{and} \quad K'(s) \geq 0 \ \quad\hbox{ for all $s\in [0,s_0]$,}
\end{equation} 
along with
\begin{equation}\label{technical condition}
\alpha'(s),\beta'(s)> 0, \ \quad\hbox{ for all $s\in (0,s_0)$}.
\end{equation}

We note that \eqref{increasing curvature} implies $K(s)\geq 0$ for $s\in (0,s_0]$.
We point out that condition \eqref{technical condition} excludes domains which are flat near the origin, 
but this case is comparatively
easier and was treated in \cite{Li-Souplet}.
Hypotheses \eqref{increasing curvature} and \eqref{technical condition} are necessary for two reasons.
On the one hand, they are needed to define a region where the parameterization $M$ is well defined.
On the other hand, when deriving the parabolic inequalities satisfied by
the auxiliary functions, they are needed to control some terms coming from the non-constant curvature.

Under the above assumptions, let us denote
\begin{equation}\label{defRadius}
R(s)=1/K(s)\in (0,\infty],\ \quad s\in [0,s_0],
\end{equation}
the radius of curvature of $\partial\Omega$ at $\gamma(s)$, and define the natural regions
\begin{equation}\label{defQGamma}
Q_\Gamma=\bigl\{(r,s)\in\mathbb{R}^2;\ 0\le r<R(s), \  0\le s\le s_0 \bigr\}
\quad\hbox{ and }\quad D_\Gamma=M(Q_\Gamma).
\end{equation}
We observe that $D_\Gamma$ is the region bordered by the four curves: $\Gamma$, the $y$-axis, the normal line at $\gamma(s_0)$
and, from above, the evolute of $\Gamma$, i.e.~the locus of the curvature centers 
\begin{equation}\label{defCurvatureCenter}
C(s)=\gamma(s)+R(s)N(s).
\end{equation}

\begin{figure}[h]
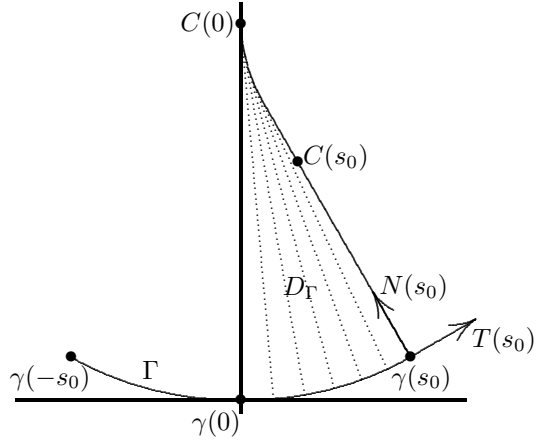

\[
\beginpicture
\setcoordinatesystem units <1cm,1cm>
\setplotarea x from -3 to 3, y from -0.2 to 5.3

\setdots <0pt>
\linethickness=1pt
\putrule from -3 0 to 3 0
\putrule from 0 -0.2 to 0 5.3

\circulararc 5 degrees from 0 0  center at 0 5
\circulararc 5 degrees from 0.43578 0.01903  center at 0.01743 4.80076
\circulararc 5 degrees from 0.85094 0.07368  center at 0.05216 4.60380
\circulararc 5 degrees from 1.24273 0.16054  center at 0.10392 4.41061
\circulararc 5 degrees from 1.60881 0.27597  center at 0.17233 4.22268
\circulararc 5 degrees from 1.94733 0.41618  center at 0.25685 4.04141

\circulararc -5 degrees from 0 0  center at 0 5
\circulararc -5 degrees from -0.43578 0.01903  center at -0.01743 4.80076
\circulararc -5 degrees from -0.85094 0.07368  center at -0.05216 4.60380
\circulararc -5 degrees from -1.24273 0.16054  center at -0.10392 4.41061
\circulararc -5 degrees from -1.60881 0.27597  center at -0.17233 4.22268
\circulararc -5 degrees from -1.94733 0.41618  center at -0.25685 4.04141

\setdots <0pt>
\setlinear 
\plot 2.25685 0.57731  0.75685 3.17539 /

\plot 0 5  0.01743 4.80076 /
\plot 0.01743 4.80076  0.05216 4.60380 /
\plot 0.05216 4.60380  0.10392 4.41061 /
\plot 0.10392 4.41061  0.17233 4.22268 /
\plot 0.17233 4.22268  0.25685 4.04141 /
\plot 0.25685 4.04141  0.75685 3.17539 /

\setdots <2pt>
\setlinear
\plot 0.43578 0.01903  0.01743 4.80076 /
\plot 0.85094 0.07368  0.05216 4.60380 /
\plot 1.24273 0.16054  0.10392 4.41061 /
\plot 1.60881 0.27597  0.17233 4.22268 /
\plot 1.94733 0.41618  0.25685 4.04141 /

\put{$\Gamma$}[cc] at -1.2 0.4
\put{$\gamma (s_0)$}[lc] at 2 0.3
\put{$\gamma (-s_0)$}[rc] at -2 0.3
\put{$\gamma (0)$}[rc] at 0 -0.3
\put{$C (0)$}[rc] at -0.075 5
\put{$C (s_0)$}[lc] at 0.831 3.23
\put{$\bullet$}[cc] at 0 5
\put{$\bullet$}[cc] at 0.75685 3.17539
\put{$\bullet$}[cc] at 0 0
\put{$\bullet$}[cc] at 2.25685 0.57731
\put{$\bullet$}[cc] at -2.25685 0.57731
\put{$D_\Gamma$}[cc] at 0.8 1.5
\put{$T (s_0)$}[cc] at 3.5 0.8
\put{$N (s_0)$}[cc] at 2.3 1.5

\setdots <0pt>
\thicklines
\arrow <10pt> [.2,.67] from 2.25685 0.57731 to 1.7569 1.4433
\arrow <10pt> [.2,.67] from 2.25685 0.57731 to 3.1229 1.0773

\endpicture
\]
\caption{Example of $\Gamma,\gamma(s),T(s),N(s)$ as in Notation \ref{notation1} and $D_\Gamma,C(s)$ defined in \eqref{defQGamma}, \eqref{defCurvatureCenter}.}
\label{Fig Notation 1}
\end{figure}

The following proposition shows that the region $D_\Gamma$ is well parametrized by $M$ and, consequently,
that one can define there the derivative $u_s$, in the parallel direction to the boundary.
Although this fact is more or less standard, we give a proof in Section~\ref{computations in flow coordinates} for convenience.

\begin{proposition}\label{naturalregion}
Let $\Omega,\Gamma,\gamma,M$ be as in Notation~\ref{notation1} and
assume \eqref{increasing curvature}, \eqref{technical condition}. 

(i) Then, the map $M$
is a diffeomorphism from $Q_\Gamma$ to $D_\Gamma$.

(ii) As a consequence, for any solution $u$ of \eqref{Hamilton-Jacobi problem}, the derivative 
$$u_s:=\frac{\partial}{\partial s}\bigl[u(M(r,s),t)\bigr]$$
is well defined in $(\overline\Omega\cap D_\Gamma)\times [0,T(u_0))$.
\end{proposition}

The following result ensures that single-point GBU occurs  
for symmetric solutions satisfying 
a monotonicity condition near the origin.

\begin{theorem}\label{local result}
Let $p>2$,
let $\Omega,\Gamma,\gamma,M$ be as in Notation~\ref{notation1} and
assume \eqref{increasing curvature}, \eqref{technical condition}.
Let $u_0\in X_+$ be a symmetric function with respect to the line $x=0$, such that~$T=T(u_0)<\infty$. 
Suppose that
\begin{equation}\label{GBUS localisation}
GBUS (u_0) \subset \gamma\bigl( -\textstyle\frac{s_0}{2}, \textstyle\frac{s_0}{2}\bigr)
\end{equation}
and that, for some $t_0\in (0,T)$, $r_0\in (0,R(s_0))$, we have
\begin{equation}\label{monotonicity condition}
 u_x, u_s < 0 \quad \text{in} \ \omega_0\times (t_0,T), \ \text{with}\ 
\omega_0 := \Omega\cap M\big( (0,r_0) \times (0, s_0)\big).
\end{equation}
Then, $GBUS(u_0)$ contains only the origin.
\end{theorem}

Hypothesis \eqref{GBUS localisation} is not difficult to guarantee. 
It is in fact satisfied whenever $u_0$ is sufficiently concentrated near the origin
(cf.~\cite{Li-Souplet} and Proposition \ref{GBU location lemma} below).
On the contrary, the hypothesis $u_s<0$ in \eqref{monotonicity condition} is in general
more difficult to verify, 
and requires assumptions of more global nature. 

The assumption $u_x<0$ in \eqref{monotonicity condition} is required by the fact that the Laplace operator does not commute with the derivative in the $s-$direction.
Therefore, we need to control a term involving $u_r$.
This can be done by writing $u_r$ as a linear combination of $u_x$ and $u_s$, see formula \eqref{r derivative in temrs of x and s derivatives}.
The term $u_x$ is obviously more tractable since the $x-$derivative does commute with the Laplace operator.
This requires the use of two auxiliary functions $J$ and $\bar J$ in the proof of this Theorem (section \ref{proof of local result}), 
the first to control $u_s$ and the second to control $u_x$. 
The derivation and analysis of the parabolic equations satisfied by $J$ and $\bar J$ necessitate
 long and technical calcutions involving boundary-fitted coordinates.

We next introduce the geometric hypotheses on the domain $\Omega$ under which we are able to construct initial data satisfying
condition \eqref{monotonicity condition}. 
To this end we set 
the following further notation,
which is motivated by moving plane arguments that we rely on.

\begin{notation}\label{notation2} For each $s\in [0,s_0]$, we denote

$\bullet$ $\Lambda_s$ the line $\gamma(s)+\mathbb{R}N(s)$

$\bullet$ $\mathcal{T}_s(\cdot)$ the symmetry with respect to $\Lambda_s$

$\bullet$ $H_s$ the half-plane at the right of the line $\Lambda_s$, i.e.:
$$H_s=\{P\in \mathbb{R}^2;\ T(s)\cdot (P-\gamma(s))>0\}.$$

$\bullet$ $\Omega_s=\Omega\cap H_s$.
\end{notation}

Using Notations \ref{notation1} and \ref{notation2}, the hypotheses that we shall assume are the following: 
\begin{equation}\label{centerout}
\overline\omega_0\subset D_\Gamma,
\quad\hbox{ where $\omega_0:=\Omega\cap D_\Gamma \cap \{y<y_0\}$, for some $y_0\in (0,\infty]$,}
\end{equation}
\begin{equation}\label{x convexity}
\begin{array}{l}
\nu_x \geq 0 \quad \text{on $\partial\Omega\cap \{x>0\}$,}
\end{array}
\end{equation}
\begin{equation}\label{thin domain condition 2 1}
\begin{array}{l}
\nu_y \geq 0 \ \text{on} \ \partial\Omega\cap\partial\omega_0 \cap \{r>0\},
\end{array}
\end{equation}
\begin{equation}\label{thin domain condition 2 2}
\mathcal{T}_{s_0}(\Omega_{s_0}) \subset\Omega,
\end{equation}
\begin{equation}\label{Omega plus}
\begin{array}{l}
\mathcal{T}_+(\Omega^+) \subset \Omega, \quad \text{where} \ \Omega^+ := \Omega\cap \{y>y_0\},  \\
\text{and $\mathcal{T}_+(\cdot)$ is the symmetry with respect to the line $y=y_0$}.
\end{array}
\end{equation}
See Figure \ref{figure second class} in section \ref{Introduction} and Figures \ref{figure moving planes 1} and \ref{figure moving planes 2} in section \ref{proof of main result} for examples of domains satisfying these hypotheses.
In view of Proposition~\ref{naturalregion}, assumption \eqref{centerout} ensures that $u_s$ is well defined in $\overline\omega_0$.
Our result reads as follows.

\begin{theorem}\label{second class theorem}
Let $p>2$ and 
let $\Omega,\gamma, s_0,\mathcal{T}_s,\Omega_s$ be as in Notations~\ref{notation1} and \ref{notation2}.
Let $D_\Gamma$ be defined by \eqref{defQGamma} and
assume \eqref{increasing curvature}, \eqref{technical condition},
\eqref{centerout}--\eqref{Omega plus}.
\begin{enumerate}
\item There exist initial data $u_0\in X_+$ such that $T(u_0)<\infty$ 
and
\begin{equation}
\label{u_0 symmetric second class}
u_0 \ \text{is symmetric with respect to the line} \ x=0,
\end{equation}
\begin{equation}
\label{u_0 decreasing hypothesis second class}
u_{0,x} \leq 0 \ \text{in} \ \Omega \cap \{x>0\} \quad \text{and} \quad u_{0,s} \leq 0 \ \text{in} \ \omega_0,
\end{equation}
\begin{equation}
\label{u_0 moving planes hyp second class}
\begin{array}{l}
u_0(P)\leq u_0(\mathcal{T}_{s_0}(P)) \quad\hbox{ for all $P \in \Omega_{s_0}$},
\end{array}
\end{equation}
\begin{equation}
\label{u_0 moving planes hyp first class}
u_0(P)\leq u_0(\mathcal{T}_+(P))\quad\hbox{ for all $P \in \Omega\cap \{y>y_0\}$}.
\end{equation}
\begin{equation}
\label{GBUS location second class}
GBUS (u_0) \subset \gamma\bigl( -\textstyle\frac{s_0}{2}, \textstyle\frac{s_0}{2}\bigr).
\end{equation}
\item For any such $u_0$, 
$GBUS(u_0)$ contains only the origin. 
\end{enumerate}
\end{theorem}

\begin{remark}\label{remark general result}
\begin{enumerate}
\item  If the domain $\Omega$ is sufficiently thin in the $y$-direction,
then the center of curvature of the boundary  lies outside $\Omega$ for all $s\in [0,s_0]$.
In that case we can consider $y_0 = +\infty$ in \eqref{centerout}  and 
conditions \eqref{Omega plus} and \eqref{u_0 moving planes hyp first class} disappear.
When this is not the case, we can
restrict $\omega_0$ to $\{y<y_0\}$, for some $y_0>0$,
in order to be able to define the boundary-fitted coordinates.
However, we then have to pay the price of assuming the reflection assumption \eqref{Omega plus}, which allows us to prove $u_y\leq 0$ on $\Omega\cap \{y=y_0\}$ by a moving planes argument.
\item Hypothesis \eqref{x convexity} implies that the domain is convex in the $x$ direction, 
and this, together with \eqref{symmetry of the domain}, 
allows one to construct solutions such that $u_x\leq 0$ in $\Omega\cap \{x>0\}$.
\item On the other hand, hypotheses \eqref{thin domain condition 2 1} and \eqref{thin domain condition 2 2} are useful to construct solutions such that $u_s<0$ in $\omega_0$.
In particular, hypothesis \eqref{thin domain condition 2 1} implies that on the upper piece of $\partial\omega_0$ which coincides
with $\partial\Omega$, $u_s$ represents the derivative in a direction pointing outside $\Omega$, and therefore $u_s\leq 0$.
Then, we prove that $u_s\leq 0$ on $\Lambda\cap \partial\omega_0$ by a moving planes argument, which can be applied only under hypothesis \eqref{thin domain condition 2 2}.
\item On $\partial\omega_0\cap \{y=y_0\}$, we prove $u_s\leq 0$ by expressing it as a linear combination of $u_x$ and $u_y$, that we can prove to be negative, see (i) and (ii) in this remark.
\end{enumerate}
\end{remark}

Observe that in figure \ref{figure second class} the domain is sufficiently thin so that we can consider $y_0=+\infty$.  
Ellipses with non-zero eccentricity, i.e.~ellipses which are not disks,
 are also examples of domains where it is possible to apply this result.
In that case, we choose $y_0$ such that the line $\{y = y_0\}$ coincides with the major axis of the ellipse.
The case of a disk is excluded since, in order to satisfy hypothesis \eqref{Omega plus}, we must consider an $y_0$ bigger or equal than the radius of curvature of the disk, but then, hypothesis \eqref{centerout} cannot hold.
However, the case of the disk can be treated using polar coordinates (see \cite{Li-Souplet}).

\begin{remark}\label{initial data remark}
Let $p>2$ and $\Omega$ be as in Theorem \ref{second class theorem},
denote $B_\rho^+:=B_\rho(0,0) \cap \{x>0\}$ and let $\rho>0$ be such that 
$$\Omega\cap B_\rho^+ \subset \omega_0,\quad \partial\Omega\cap B_\rho^+\subset \gamma(0,s_0/2).$$
It follows from Theorem \ref{second class theorem} and Proposition \ref{GBU location lemma} below that 
$T(u_0)<\infty$ and $GBUS (u_0)=\{(0,0)\}$ whenever $u_0\in X_+$ for instance satisfies
\eqref{u_0 symmetric second class}, \eqref{u_0 decreasing hypothesis second class} and
\begin{eqnarray*}
&&{\rm supp}(u_0)\subset \overline\Omega\cap \overline B_{\rho/2}, \\
&&\|u_0\|_\infty\le C_2, \\
&&\inf_{\tilde B_\eps} u_0\ge C_1\eps^k\quad\hbox{with $\tilde B_\eps=B_{\eps/2}(0,\eps)$, \ for some $\eps\in(0,\rho/2)$,  }
\end{eqnarray*} 
where  $C_1(p)>0$ and $C_2(p,\Omega,\rho)>0$.
Moreover, initial data satisfying these assumptions can be easily constructed.
See the proof of Theorem \ref{second class theorem}(i) for details.
\end{remark}

The outline of the rest of the paper is as follows.
In section \ref{computations in flow coordinates} we give some basic computations and notation on
the ``boundary-fitted'' curvilinear coordinate system
and we give the proof of Proposition \ref{naturalregion}.
In section \ref{previous results} we give some useful preliminary results, concerning nondegeneracy and localization of GBU as well as a Serrin type corner lemma.
Theorems \ref{local result} and \ref{second class theorem} are respectively proved in sections \ref{proof of local result}
and~\ref{proof of main result}.
Finally in section \ref{reduction section}, we deduce Theorems \ref{ellipses theorem} and \ref{almost flat theorem} from Theorem \ref{second class theorem}.

\section{Preliminary results I: basic computations in boundary-fitted curvilinear coordinates}\label{computations in flow coordinates}

In this section we give some basic computations in the coordinate system given by the map $M$ in \eqref{flow coordinates}.
Here $\Omega$ and $\Gamma$ are as in Notation \ref{notation1}
and we assume conditions \eqref{increasing curvature} and \eqref{technical condition}.
By Proposition \ref{naturalregion}, that we will prove at the end of this section, $M$ is a diffeomorphism from $Q_\Gamma$ to $D_\Gamma$,
where $Q_\Gamma$ and $D_\Gamma$ are defined in \eqref{defQGamma}. 
To facilitate the change of coordinates throughout the paper, we adopt the following notation and conventions.

\goodbreak 
\begin{notation}\label{notation3} 
For any function $\psi(x,y)$ defined on (a part of) $D_\Gamma$, we express $\psi$ in terms of the variables $(r,s)$ by setting
$$\tilde{\psi}:=\psi\circ M,$$
i.e. $\tilde{\psi} (r,s) = \psi (M(r,s))$ for $(r,s)\in Q_\Gamma$.
The derivatives with respect to the variables $(r,s)$ of a function $\psi \ =\psi(x,y)\in C^1(D_\Gamma)$ are then defined by
\begin{equation}\label{definition psirstilde}
\psi_r:=\tilde\psi_r,\qquad \psi_s:=\tilde\psi_s.
\end{equation}
Similarly, for any function $\varphi(r,s)$ defined on (a part of) $Q_\Gamma$, we denote 
$$\hat \varphi=\varphi\circ M^{-1}.$$
In the rest of the paper, for any functions $\psi=\psi(x,y)$ and $\varphi=\varphi(r,s)$, when no risk of confusion arises, 
we will drop the tilde and the hat and will just write $\psi(r,s)$ in place of $\tilde{\psi} (r,s)$ and $\varphi(x,y)$ in place of $\hat{\varphi} (x,y)$.

Also, the gradient and the Laplacian operators will always be understood as
$$\nabla \psi=(\psi_x, \psi_y)$$
and
$$\Delta\psi=div(\nabla \psi)=\psi_{xx}+\psi_{yy},$$
either as functions of $(x,y)$, or as functions of $(r,s)$ (i.e., implicitly considering 
$(\nabla \psi)\circ M$ and $(\Delta \psi)\circ M$).
\end{notation}

According to the chain rule, we have
\begin{equation}\label{definition psirs}
\psi_r = \nabla \psi (M(r,s)) \cdot N(s) \quad \text{and} \quad
\psi_s = \nabla \psi (M(r,s)) \cdot (\gamma'(s) + rN'(s)).
\end{equation}  
Using \eqref{K definition 2} and \eqref{K definition 3}, we obtain
$$
N'(s) = -K(s)(\alpha'(s),\beta'(s)) = -K(s) T(s),
$$
and then, we can rewrite \eqref{definition psirs} as 
\begin{equation}\label{r and s derivatives}
\psi_r = \nabla \psi \cdot N(s) \quad \text{and} \quad
\psi_s =  (1-rK(s)) \nabla\psi \cdot T(s).
\end{equation}
Note that 
\begin{equation}\label{rKs}
1-rK(s)>0\quad\hbox{ in $D_\Gamma$}
\end{equation}
owing to  \eqref{defRadius}, \eqref{defQGamma}.
Since the vectors $N(s)$ and $T(s)$ are orthonormal, we then have
\begin{equation}\label{gradient flow coordinates}
\begin{array}{l}
\nabla \psi (r,s) \,\equiv (\nabla \psi)\circ M=\psi_r N(s) + \dfrac{\psi_s}{1-rK(s)} T(s),
\end{array}
\end{equation}
as well as
\begin{equation}\label{diff op in flow coord2}
\nabla \psi \cdot \nabla \varphi = \psi_r \varphi_r + \dfrac{\psi_s \varphi_s}{(1-r K)^2}.
\end{equation} 
 
We next recall two alternative expressions for the function curvature of the boundary $K(s)$.  
Since $\gamma(s) = (\alpha(s),\beta(s))$ is an arclength parametrization, we have
$$\alpha'(s)\alpha''(s) + \beta'(s)\beta''(s) = \dfrac{(\alpha'(s)^2+\beta'(s)^2)'}{2} = 0,$$
and then we have $\alpha'(s) \alpha''(s) = -\beta'(s)\beta''(s)$. Using this identity, we can obtain
\begin{equation}\label{K definition 2}
K(s) = \alpha'(s)\beta''(s) - \beta'(s)\alpha''(s) = \alpha'(s)\beta''(s) + \dfrac{\beta'(s)\beta''(s)}{\alpha'(s)} = \dfrac{\beta''(s)}{\alpha'(s)},
\end{equation}
and in a similar way, recalling \eqref{technical condition}, we obtain
\begin{equation}\label{K definition 3}
K(s) = - \dfrac{\alpha''(s)}{\beta'(s)}, \ \quad s\ne 0.
\end{equation}

Now, we give some further identities relating the derivatives in boundary-fitted coordinates with the derivatives in cartesian coordinates. 
As we will see in our proofs, we have particular interest in expressing, when possible, 
$\psi_r$ as a linear combination of $\psi_x$ and $\psi_s$.
In the following computations, and without risk of confusion, we omit the dependence on $s$ of the functions $K,\alpha',\beta'$.
In view of \eqref{r and s derivatives}, we have
\begin{equation}\label{psi_r psi_s}
\begin{array}{rcl}
\psi_r &=& -\beta'\psi_x + \alpha'\psi_y, \\
 \noalign{\vskip 1mm}
\dfrac{\psi_s}{1-rK} &=& \alpha' \psi_x + \beta'\psi_y.
\end{array}
\end{equation}
Then, recalling \eqref{technical condition}, we obtain the identity
\begin{equation}\label{r derivative in temrs of x and s derivatives}
\psi_r = -\dfrac{1}{\beta'}\psi_x + \dfrac{\alpha'}{\beta'}\dfrac{\psi_s}{1-rK}, \ \quad s\ne 0.
\end{equation}
We note that it is possible to write $\psi_r$ as a linear combination of $\psi_x$ and $\psi_s$
only when $\beta'(s)\neq 0$
(i.e.,~$s\ne 0$). 
This makes sense since, if $\beta'(s)=0$, then $\psi_x = \psi_s$ and $\psi_r$ is the derivative in the $y$ direction, 
which is then orthogonal to the $x$ and $s$ directions.

The next result is a very useful expression of the Laplacian in flow coordinates.

\begin{proposition}\label{Lap in flow coord}
(i) Let $\psi=\psi(x,y) \in C^2(D_\Gamma)$. We have
\begin{equation}\label{diff op in flow coord}
\Delta \psi\equiv (\Delta \psi)\circ M= \psi_{rr} - \dfrac{K}{1-r K} \psi_r + \dfrac{1}{(1-r K)^2} \psi_{ss} + \dfrac{r K'}{(1-r K)^3} \psi_s,
\ \quad  (r,s)\in Q_\Gamma.
\end{equation}
(ii) If $\varphi=\varphi(r,s) \in C^2(Q_\Gamma)$, then 
$\Delta \varphi\equiv [\Delta (\varphi\circ M^{-1})]\circ M$ is also given by \eqref{diff op in flow coord}
with $\psi$ replaced by $\varphi$.
\end{proposition} 

\begin{proof}
(i) For $\varphi=\varphi(r,s)$ recall the notation $\hat \varphi=\hat \varphi(x,y):=\varphi\circ M^{-1}$.
For any $\psi=\psi(x,y)\in C^2(D_\Gamma)$, using \eqref{gradient flow coordinates}, we obtain 
$$\nabla \psi  = \widehat{\psi_r} \,\hat N + \dfrac{\widehat{\psi_s}}{1-\hat r \hat K} \,\hat T
\quad\hbox{ in $D_\Gamma$.}$$
It follows that
\begin{equation}\label{Delta psi computation}
\begin{array}{rcl}
\Delta \psi = \text{div} (\nabla \psi) &=& \nabla (\widehat{\psi_r})\cdot \hat N+ \widehat{\psi_r} \ \text{div}\,\hat N + \dfrac{1}{1-\hat r\hat K} \nabla (\widehat{\psi_s}) \cdot \hat T \\
 \noalign{\vskip 1mm}
& & + \widehat{\psi_s} \nabla \left( \dfrac{1}{1-\hat r \hat K}\right) \cdot \hat T + \dfrac{1}{1-\hat r\hat K} \widehat{\psi_s} \ \text{div}\,\hat T.
\end{array}
\end{equation}
By  \eqref{gradient flow coordinates}, we have
$$[(\nabla \varphi)\circ M]\cdot N=\varphi_r
 \quad \text{and} \quad
 (1-rK) [(\nabla\varphi)\circ M]\cdot T=\varphi_s,$$
hence
$$\nabla \varphi\cdot \hat N=\widehat{\varphi_r}\equiv (\varphi\circ M)_r\circ M^{-1} 
 \quad \text{and} \quad
  (1-\hat r\hat K) (\nabla \varphi\cdot \hat T)=\widehat{\varphi_s}\equiv (\varphi\circ M)_s\circ M^{-1}.  $$
Using this with $\varphi=\widehat{\psi_r}$, we can thus identify
\begin{equation}\label{identification1}
 \nabla (\widehat{\psi_r})\cdot \hat N\equiv  \nabla ((\psi\circ M)_r\circ M^{-1})\cdot \hat N
=(\psi\circ M)_{rr} \circ M^{-1} \equiv \psi_{rr} \circ M^{-1},
\end{equation}
\begin{equation}\label{identification2}
(1-\hat r\hat K) \nabla (\widehat{\psi_s})\cdot \hat T\equiv (1-\hat r\hat K) \nabla ((\psi\circ M)_s\circ M^{-1})\cdot \hat T
=(\psi\circ M)_{ss} \circ M^{-1} \equiv \psi_{ss} \circ M^{-1}
\end{equation}
and
\begin{equation}\label{identification3}
\nabla \left( \dfrac{1}{1-\hat r\hat K}\right) \cdot \hat T=\dfrac{1}{1-\hat r\hat K}\biggl[\dfrac{1}{1-rK}\biggr]_s\circ M^{-1}
=\dfrac{rK'}{(1-rK)^3}\circ M^{-1}.
\end{equation}

On the other hand, since  
$N(s) = -\beta'(s)(1,0)+\alpha'(s)(0,1)$, we have
$$\text{div} (\hat N) = -\nabla\widehat{\beta'} \cdot (1,0) + \nabla \widehat{\alpha'} \cdot (0,1).$$
Applying \eqref{gradient flow coordinates} with $\psi=\widehat{\beta'}$ and  $\psi=\widehat{\alpha'}$, we obtain
\begin{eqnarray}
[\text{div} (\hat N)]\circ M
&=& -\dfrac{\beta''}{1-rK}T(s)\cdot (1,0) + \dfrac{\alpha''}{1-rK}T(s)\cdot (0,1)\notag \\ 
&=& -\dfrac{\beta''\alpha'}{1-rK}  + \dfrac{\alpha''\beta'}{1-rK} = -\dfrac{K}{1-rK}. \label{divN}
\end{eqnarray}
Similarly, since $T(s) = \alpha'(s) (1,0) + \beta'(s) (0,1)$, hence
$$\text{div} (\hat T) = \nabla\widehat{\alpha'} \cdot (1,0) + \nabla \widehat{\beta'} \cdot (0,1),$$
we have
\begin{eqnarray}
[\text{div} (\hat T)]\circ M
&=& \dfrac{\alpha''}{1-rK} T(s)\cdot (1,0) + \dfrac{\beta''}{1-rK}T(s)\cdot (0,1)\notag \\ 
&=& \dfrac{\alpha''\alpha'}{1-rK} + \dfrac{\beta''\beta'}{1-rK} = \dfrac{(\alpha'^2+\beta'^2)'}{2(1-rK)} = 0. \label{divT}
\end{eqnarray}
Finally, plugging \eqref{identification1}--\eqref{divT} 
in \eqref{Delta psi computation}, we obtain \eqref{diff op in flow coord}.

(ii) It suffices to apply assertion (i) to $\psi:=\varphi\circ M^{-1}$, using \eqref {definition psirstilde} and the fact that
$\tilde\psi\equiv\psi\circ M=\varphi$.
\end{proof}

We end this section with the proof of Proposition \ref{naturalregion}.

\begin{proof}[Proof of Proposition \ref{naturalregion}]
It suffices to show assertion (i).
We first establish the injectivity of $M$ on $Q_\Gamma$.
Let $C(s)=\gamma(s)+R(s)N(s)$ be the center of curvature.
We note that $D_\Gamma$ can be written as the union of half-open segments:
$$D_\Gamma=\displaystyle\bigcup_{s\in [0,s_0]} \Sigma(s),\quad\hbox{ where $\Sigma(s)=[\gamma(s),C(s))$.}$$
To show the injectivity, it suffices to verify that for any $0\le s_1<s_2\le s_0$,
the segments $\Sigma(s_1)$ and $\Sigma(s_2)$ do not intersect.
This amounts to showing that $\Sigma(s_2)$ lies entirely in the open half-plane to the right of the line 
$\Lambda_{s_1}$, defined as in Notation \ref{notation2}, which is the line containing the segment $\Sigma(s_1)$.
This half-plane is defined by the inequality
$$T(s_1)\cdot (x-\gamma(s_1))>0, \quad \text{with} \ x\in \mathbb{R}^2.$$
Considering the extremes of the segment $\Sigma(s_2)$, this is thus equivalent to
\begin{equation}\label{condprop1}
T(s_1)\cdot (\gamma(s_2)-\gamma(s_1))>0 \quad\hbox{ and }\quad T(s_1)\cdot (C(s_2)-\gamma(s_1))\ge 0.
\end{equation}

To show \eqref{condprop1}, using $\gamma'(s)=T(s)$ and \eqref{technical condition}, we first compute
$$\frac{d}{ds}\Big( T(s_1)\cdot (\gamma(s)-\gamma(s_1))\Big) = T(s_1)\cdot T(s) > 0,\ \quad s_1<s\le s_0,$$
hence the first inequality in \eqref{condprop1} follows.
On the other hand, using $N'(s)=-K(s)T(s)$, \eqref{increasing curvature} and \eqref{technical condition}, we get
$$\frac{d}{ds}\Big(T(s_1)\cdot N(s)\Big)=-K(s) T(s_1)\cdot T(s)\le 0,$$
Since $T(s_1)\cdot N(s_1)=0$, we deduce that
\begin{equation}\label{condprop2}
T(s_1)\cdot N(s) \le 0,\ \quad s_1<s\le s_0.
\end{equation}
Also, using $\gamma'(s)=T(s)$ and $N'(s)=-K(s)T(s)$, we have
$$C'(s)=(1-K(s)R(s))T(s)+R'(s)N(s)=R'(s)N(s).$$
Since $R'(s)\le 0$ due to \eqref{increasing curvature}, it follows from \eqref{condprop2} that
$$\frac{d}{ds}\Big( T(s_1)\cdot (C(s)-C(s_1))\Big) =R'(s)T(s_1)\cdot N(s)\ge 0,\ \quad s_1<s\le s_0,$$
hence, it follows from $\gamma (s_1) =  C(s_1) - R(s_1)N(s_1)$ that
$$T(s_1)\cdot (C(s_2)-\gamma(s_1)) = T(s_1)\cdot (C(s_2)-C(s_1))\ge 0,$$
which guarantees the second inequality in \eqref{condprop1}.
This completes the proof of the injectivity.

To prove that $M$ is a diffeomorphism from $Q_\Gamma$ to $D_\Gamma=M(Q_\Gamma)$,
it thus suffices to show that the Jacobian of $M$ does not vanish in $Q_\Gamma$.
For all $(r,s)\in Q_\Gamma$, using $\gamma'=T$ and $N'=-KT$ again, we compute
$${\rm Jac}_M(r,s)=det\Bigl(\frac{\partial M}{\partial r},\frac{\partial M}{\partial s}\Bigr)
=det\Bigl(N,(1-Kr)T\Bigr)=K(s)r-1<0,$$
since $r<R(s)=1/K(s)$, and the conclusion follows.
\end{proof}

\section{Preliminary results II: Nondegeneracy and localization of GBU and corner lemma}\label{previous results}

In this section we give three preliminary results that we use in the proofs of Theorems \ref{local result}
and \ref{second class theorem}.
We start with the following nondegeneracy lemma, proved in \cite{Li-Souplet}, which implies that, at any 
 gradient blow-up point, the estimate \eqref{bernstein estimate 1} is essentially optimal in the normal direction to the boundary. 

\begin{lemma}\label{nondegeneracy lemma}
Let $\Omega\in \mathbb{R}^2$ be a smoothly bounded domain and $x_0\in \partial\Omega$.
There exists $c_0 = c_0(p)$ such that, if
$$
u\leq c_0 \delta^{(p-2)/(p-1)} (x,y) \quad  \text{in} \ (B_\rho (x,0) \cap \Omega) \times [0,T),
$$
for some $\rho >0$, then $x_0$ is not a gradient blow-up point.
\end{lemma}

We observe that, as a consequence of this Lemma,  
if $x_0\in \partial\Omega$ is a gradient blow-up point, then we must have
$$
\limsup_{x\to x_0, t\to T} u(x,y,t) \delta^{-(p-2)/(p-1)} (x,y) \geq c_0(p),
$$ 
In view of \eqref{bernstein estimate 1}, it follows in particular  that
$$
\limsup_{x\to x_0, t\to T} u_\nu (x,y,t) \delta^{1/(p-1)}(x,y) \in (0,\infty).
$$ 
where $u_\nu$ is the derivative of $u$ in the outward normal direction to the boundary.

The second preliminary result is the following proposition, which
provides a sufficient condition on the initial data $u_0$ under which the solution blows up, with GBU set concentrated near an arbitrary given point. 
The idea of proof is based on that of \cite[Theorem~1.1]{Li-Souplet}, where a more particular example of initial data was given.

\begin{proposition}\label{GBU location lemma}
Let $p>2$, $\Omega\subset\R^2  $ be a smoothly bounded domain and let $x_0\in\partial\Omega$ and $\rho>0$.
There exist constants $C_1(p)>0$ and $C_2(p,\Omega,\rho)>0$ with the following property:

If for some $\eps>0$ such that $\tilde{B}_\eps := B(x_0 + \eps\nu(x_0), \eps)\subset \Omega$,
$u_0\in X_+$ satisfies
\begin{eqnarray}
\label{supp u_0}
&&{\rm supp}(u_0)\subset \overline\Omega\cap \overline B(x_0,\rho/2), \\
\label{L infinity u_0}
&&\|u_0\|_\infty\le C_2, \\
\label{inf u_0}
&&\inf_{\tilde B_{\eps/2}} u_0\ge C_1\eps^k,\quad\hbox{ with $\tilde B_{\eps/2}:= B(x_0+\eps\nu(x_0),\eps/2)$,}
\end{eqnarray} 
where $k=(p-2)/(p-1)$, then  
$$\hbox{$T(u_0)<\infty$\quad and \quad $GBUS(u_0) \subset B_\rho(x_0) \cap\partial\Omega$.}$$
\end{proposition}

\begin{proof}
We divide the proof into two steps.

\textit{Step 1: $\nabla u$ blows up in finite time.}
The idea here is to use the auxiliary function introduced in \cite{Li-Souplet} as subsolution. 
Let $\varphi\in C^\infty ([0,\infty))$ be a function satisfying
$$
\varphi'\leq 0, \qquad \varphi(r) = 1, \ \text{for} \ r\leq 1/4, \qquad \varphi(r) = 0, \ \text{for} \ r\geq 1/2.
$$
Consider the following problem:
\begin{equation}\label{GBU problem in B_1}
\begin{array}{cl}
v_t - \Delta v = |\nabla v|^p, & x\in B_1(0), \ t>0, \\
v(x,t) = 0, & x\in \partial B_1(0), \ t>0, \\
v(x,0) = \phi (x) := C_1 \varphi(|x|), & x \in B_1(0).
\end{array}
\end{equation}
By \cite[Thm 4.2]{QS} (see also \cite[Prop. 7.1]{Souplet-Zhang}), there exists $C_0 = C_0(p)$ such that,
if $\|\phi\|_1 \geq C_0$, then $T(\phi) <\infty$.
Therefore, we have $T(\phi)<\infty$ whenever $C_1$ is bigger than some constant depending on $p$.
We now use the scale invariance of  the equation.
Namely we consider the rescaled function
$$
v_\eps (x,t) = \eps^k v\left( \eps^{-1} |x-\tilde{x}_0|, \eps^{-2}t\right),
$$
where $\tilde{x}_0 = x_0 + \eps \nu(x_0)$.
Then $v_\eps$ solves \eqref{GBU problem in B_1} in $\tilde{B}_\eps \subset\Omega$.

Since we have
$$
v_\eps (x,0) = \eps^k C_1 \varphi (\eps^{-1}|x-\tilde{x}_0|) \leq \eps^k C_1 \ \text{in} \  \tilde{B}_{\eps/2},
$$
and $v_\eps (x,0) = 0$ in $\tilde{B}_\eps \setminus \tilde{B}_{\eps/2}$,
we can use \eqref{inf u_0}, together with the comparison principle to get
$$
u\geq v_\eps \quad
\text{in} \ \tilde{B}_\eps \times (0,\tilde{T}), \quad \text{where} \ \tilde{T} = \min (T(u_0),T_\eps) \ \text{and} \ T_\eps = \eps^2 T(\phi).
$$
Now we observe that $\tilde{B}_\eps$ is tangent to $\partial\Omega$ at $x_0$, so we deduce
$$
- \dfrac{\partial u}{\partial \nu} (x_0,t) \geq - \dfrac{\partial v_\eps}{\partial \nu} (x_0,t), \qquad 0<t<\tilde{T}.
$$
On the other hand, as a consequence of the maximum principle applied to $\nabla v$
(see e.g. \cite[Prop. 40.3]{QS}), we know that
$$
\displaystyle\max_{t\in[0,\tau]} \| \nabla v(\cdot,t) \|_\infty
= \max \left( \|\nabla v(\cdot,0)\|_\infty, \ \displaystyle\max_{\partial B_1(0) \times [0,\tau]} \left(-\dfrac{\partial v}{\partial\nu}\right)\right), \qquad
0<\tau <T(\phi).
$$
Since $v$ is radially symmetric, it follow that
$$
\displaystyle\limsup_{t\to T_\eps} \dfrac{\partial v_\eps}{\partial \nu} (x_0,t) = \infty,
$$
hence $T(u_0) \leq T_\eps < \infty$.

\textit{Step 2: No GBU on $\partial\Omega\setminus B_\rho(x_0)$.}
For $\rho>0$, consider a cut-off function $h\in C^\infty ([0,\infty))$ satisfying
$$
h'\leq 0, \qquad
h(r) = 1,\quad \text{for} \ r \leq \rho/2,
\qquad
h(r) = 0, \quad \text{for} \ r\geq 3\rho/4.
$$
Now, let $h_{x_0}$ be the function in $\overline{\Omega}$ defined by
$$ h_{x_0} (x) := h(|x-x_0|). $$

Let $\psi= \psi_{x_0}$ be the unique classical solution of the linear elliptic problem
\begin{equation}\label{initial data lemma elliptic problem}
\left\lbrace \begin{array}{cc}
-\Delta \psi (x)= 1, & x\in \Omega, \\
\noalign{\vskip 1mm}
\psi (x) = h_{x_0}(x), & x\in \partial\Omega.
\end{array}
\right. 
\end{equation}

We claim that there exists $c_1>0$, independent of $x_0$, satisfying
$$
\psi(x) \geq c_1, \qquad \text{for all} \ x\in \Omega \cap B(x_0,\rho/2).
$$
We can prove this claim by using a contradiction and compactness argument.
Suppose there exists a sequence $\{ x_i\}_{i\in\mathbb{N}} \subset \partial \Omega$ such that
\begin{equation}\label{min psi to 0}
\min_{\overline{\Omega \cap B(x_i,\rho/2)}} \psi_{x_i} (x) \to 0, \quad \text{as} \ i\to \infty,
\end{equation}
where $\psi_{x_i}$ is the solution of \eqref{initial data lemma elliptic problem} with boundary data $h_{x_i}$.
Since $\partial\Omega$ is compact, we can suppose, by extracting a subsequence, that $x_i$ converges to some $x_{\infty}\in \partial\Omega$.

Now fix some $\alpha\in (0,1)$ and observe that, by the construction of $h_{x_0}$ above, there exists $C>0$, independent of~$i$, such that
$\|h_{x_i}\|_{C^{2+\alpha}(\overline\Omega)} \leq C$, and therefore $\|\psi_{x_i}\|_{C^{2+\alpha}(\overline\Omega)} \leq C'(C,\Omega)$
by interior-boundary elliptic Schauder estimates (see Theorem 47.2 (ii) in \cite{QS}).
Hence, as $h_{x_i}$ converges to $h_{x_\infty}$ in $C^{2+\alpha}(\overline\Omega)$, 
by compact embeddings and uniqueness for problem \eqref{initial data lemma elliptic problem}, we can deduce that 
$\psi_{x_i}$ converges to $\psi_{x_\infty}$ in $C^{2}(\overline\Omega)$.
It then follows from \eqref{min psi to 0} that $\psi_{x_\infty}$ vanishes somewhere in $\overline{\Omega\cap B(x_\infty,\rho/2)}$.

Since $h_{x_\infty} (x) =1$ in $\overline B(x_\infty,\rho/2)$, and then $\psi_{x_\infty} (x) = 1$ in $\partial\Omega\cap  \overline B(x_\infty,\rho/2)$, we deduce that $\psi_{x_\infty}$ vanishes somewhere in the interior of $\Omega$, contradicting the strong maximum principle. The claim is then proved.

On the other hand, applying elliptic estimates again,
there exists $\tilde C = \tilde C(\rho, \Omega)>0$ such that $\|\nabla \psi\|_{\infty}\leq \tilde C$.
Choosing $c_2=\tilde C^{-p/(p-1)}$, we then have
$\|\nabla \psi\|_{\infty}^{-p/(p-1)}\ge c_2$, hence
$$
-\Delta (c_2 \psi) = c_2 \geq |\nabla (c_2\psi)|^p, \qquad \text{in} \ \Omega.
$$
And by \eqref{L infinity u_0} with $C_2=c_1c_2$, we have
$$
c_2\psi \geq c_2 c_1 \geq u_0, \qquad \text{in}\  B(x_0,\rho/2),
$$
hence, using \eqref{supp u_0}, we get $c_2\psi \geq u_0$ in $\Omega$.
By the comparison principle, it follows that $u\leq c_2\psi$ in $\Omega\times (0,T(u_0))$.
Therefore, since $\psi=0$ on $\partial \Omega \setminus B_{3\rho/4}(x_0)$, we have
$$
0\leq - \dfrac{\partial u}{\partial \nu}  \leq -c_2 \dfrac{\partial \psi}{\partial\nu}  \leq C, \quad \text{on} \ (\partial \Omega \setminus B_{3\rho/4}(x_0))\times (0,T(u_0)). 
$$
The conclusion then follows from Lemma \ref{nondegeneracy lemma}.
\end{proof}

We conclude this section with a parabolic version of  ``Serrin's corner Lemma'',
 adapted to our parabolic problem and domain.

\begin{lemma}\label{Serrin corner lemma 2}
Let $p>2$ and $u_0 \in X_+$, let $\Omega, \Gamma, \gamma, M$ be as in Notation \ref{notation1}.
and assume  \eqref{increasing curvature} and \eqref{technical condition}. 
Suppose that there exist $t_0\in(0,T)$, $s_1\in (0,s_0)$, $r_0>0$ and $c_1>0$ such that
$$\omega_0 :=M((0,r_0)\times (0,s_0)) \subset \Omega\cap D_\Gamma,$$
\begin{equation}\label{Serrin monotony hyp 2}
u_x <0 \quad \text{in} \ \omega_0\times (t_0,T)
\end{equation}
and
\begin{equation}\label{Serrin Hopf hyp 2}
u_x \leq -c_1 r \quad \text{on} \ (0,r_0)\times\{s_1\}\times(t_0,T).
\end{equation}
Then, for any fixed $r_1\in(0,r_0)$ and $t_1\in (t_0,T)$, 
there exists $\tilde{c}_1>0$ such that
$$
u_x (r,s,t_1) \leq -\tilde{c}_1 r s \quad \text{in} \ \omega_1,
$$
where $\omega_1 := M((0,r_1)\times (0,s_1)).$
\end{lemma}

\begin{proof}
We fix a nontrivial smooth function $\phi\geq 0$ on $[0,r_0]$, with $\text{supp} (\phi) \subset \subset (0,r_0)$ and another smooth function $\psi$ on $[0,s_1]$ such that
$$
\psi = 0 \quad \text{on} \ \left[0,\dfrac{s_1}{2}\right], \quad \psi (s_1) = 1, \quad
\psi',\psi''\geq 0.
$$ 
Fix a constant $M>0$ such that
\begin{equation}\label{M 2}
M\geq \dfrac{K}{1-rK} + p|\nabla u|^{p-1}, \quad M\geq \dfrac{rK'}{(1-rK)^3} + \dfrac{p|\nabla u|^{p-1}}{1-rK},
\quad\hbox{ in $\omega_0\times (t_0,t_1]$.}
\end{equation}
Next, fix $t_2 \in (t_0,t_1)$ and let $v,V$ be the respective global solutions of
\begin{equation*}
\begin{array}{ll}
v_t-v_{rr} = - M|v_r|, & r\in(0,r_0),\ t>t_2, \\
 \noalign{\vskip 1mm}
v(0,t) = v(r_0,t) = 0, & t>t_2, \\
 \noalign{\vskip 1mm}
v(r,t_2) = \phi(r), & r\in [0,r_0],
\end{array}
\end{equation*}
and
\begin{equation}\label{equation V 2}
\begin{array}{ll}
V_t - V_{ss} = - M V_s, & s\in (0,s_1),\ t>t_2, \\
 \noalign{\vskip 1mm}
V(0,t)=0, \  V(s_1,t) = 1, & t>t_2, \\
 \noalign{\vskip 1mm}
V(s,t_2) = \psi(s), & s\in[0,s_1].
\end{array}
\end{equation}
By the maximum principle we have $v\ge 0$, $0\leq V\leq 1$, and $V_s\geq 0$. 
Also, by \eqref{equation V 2}, we deduce that $V_{ss}(s,t)\geq 0$, for $s\in \{0,s_1\}$ and $t>t_2$.
Since $\psi''\geq 0$, it follows from the maximum principle that $V_{ss}\geq 0$, for $s\in(0,s_1)$, $t>t_2$.
Moreover, by Hopf's lemma, for some $c_0>0$, we have
\begin{equation}\label{v V Hopf 2}
v(r,t_1) \geq c_0 r \ \text{in} \ (0,r_1 ), \qquad V(s,t_1) \geq c_0 s \ \text{in} \ (0,s_1).
\end{equation}

Let then $z(r,s,t) = v(r,t) V(s,t)$. We compute
\begin{eqnarray*}
z_t - z_{rr} - \dfrac{1}{(1-rK)^2} z_{ss} &=& V(v_t - v_{rr}) + v\left(V_t - \dfrac{1}{(1-rK)^2}V_{ss}\right) \\
&\leq & -M|z_r| - M|z_s|.
\end{eqnarray*}
Hence, using \eqref{gradient flow coordinates},  Proposition~\ref{Lap in flow coord} and the choice of $M$ in \eqref{M 2}, we obtain
\begin{equation}\label{parab ineq serrin 1 2}
\begin{array}{rcl}
z_t - \Delta z &=& z_t-z_{rr} + \dfrac{K}{1-rK}z_r - \dfrac{1}{(1-rK)^2} z_{ss} - \dfrac{rK'}{(1-rK)^3}z_s \\
 \noalign{\vskip 1mm}
& \leq & -\left(M - \dfrac{K}{1-rK}\right) |z_r| - \left(M - \dfrac{rK'}{(1-rK)^3}\right) |z_s| \\
 \noalign{\vskip 1mm}
& \leq & p|\nabla u|^{p-2} \nabla u\cdot \nabla z.
\end{array}
\end{equation}
On the other hand,$W :=-u_x$  satisfies
\begin{equation}\label{parab ineq serrin 2 2}
W_t - \Delta W = p|\nabla u|^{p-2}\nabla u\cdot \nabla W.
\end{equation}

For $\mu\in (0,1)$ small enough, due to \eqref{Serrin monotony hyp 2}, together with $\text{supp} (\phi) \subset\subset (0,r_0)$ and $\psi \equiv 0$ in $[0,s_1/2]$, we have
$$
-u_x(r,s,t_2) \geq \mu \phi(r) \psi(s) = \mu z(r,s,t_2) \quad \text{in} \ \omega_1. 
$$
Moreover, for possibly smaller $\mu>0$, using \eqref{Serrin Hopf hyp 2}, we see that
$$
-u_x (r,s_1,t) \geq c_1 r \geq \mu v(r,t) = \mu z(r,s_1,t), \quad
r\in(0,r_0), \ t\in [t_2,t_1].
$$
Since $z=0$ on the rest of the lateral boundary of $ \omega_1\times[t_2,t_1]$ 
(i.e. for $r\in\{r_0,1\}$ or $s=0$), it follows from \eqref{parab ineq serrin 1 2}, \eqref{parab ineq serrin 2 2}, the comparison principle and \eqref{v V Hopf 2} that
$$
-u_x (r,s,t_1) \geq \mu v(r,t_1) V(s,t_1) \geq \tilde{c}_1 rs \quad \text{in} \  \omega_1,
$$
with $\tilde{c}_1=\mu c_0^2$.
\end{proof}

\section{Proof of Theorem \ref{local result}}\label{proof of local result}

\subsection{Auxiliary parabolic inequalities}

Theorem \ref{local result} will be proved by using the techniques introduced in \cite{Li-Souplet}, that 
we here have to modify in a nontrivial way in order to adapt the method to the boundary with non constant curvature. 
These techniques are based on a Friedman-McLeod-type argument \cite{FM}, which is very useful for solutions which are monotone in some sense. In our case, 
this monotonicity follows from the hypothesis \eqref{monotonicity condition}.

Recall Notation \ref{notation3} and \eqref{rKs}.
Let $\sigma\in \bigl(0,\frac{1}{2(p-1)}\bigr)$ be fixed. 
For given $\eta\in (0,s_0/2)$, we consider the auxiliary functions
\begin{equation}\label{J def}
J = \dfrac{u_s}{1-rK(s)} + c(s) d(r) F(u)
\end{equation}
and
\begin{equation}\label{tilde J def}
\bar J = u_x + \bar c(s) d(r) F(u),  
\end{equation}
defined in $(D_\Gamma \cap \Omega)\times (0,T)$, where $D_\Gamma$ is given in \eqref{defQGamma} and
\begin{equation}\label{cdF def}
\begin{array}{ll}
F(u) = u^q, &  1<q<2, \\ 
 \noalign{\vskip 1mm}d(r) = r^{-\gamma}, & \gamma = (1-2\sigma)(q-1), \\
 \noalign{\vskip 1mm}
 c(s) = k (s-\eta), & k \  \in (0,1), \\
 \noalign{\vskip 1mm}
\bar c(s) = ks, &
\end{array}
\end{equation}
where $k,\gamma$ will be taken small (i.e., $q$ close to $1$).

We start with a Lemma giving the equation satisfied by the first part of $J$.

\begin{lemma}\label{parabolic eq for u_s lemma}
Let $\Omega,\Gamma,\gamma,M$ be as in Notation \ref{notation1} and assume \eqref{increasing curvature},\eqref{technical condition}.
Then, the function $w = \dfrac{u_s}{1-rK}$ satisfies
\begin{equation}\label{parabolic ineq second class}
w_t - \Delta w = a_w w + b_w\cdot \nabla w + \dfrac{K'}{(1-rK)^3} \frac{1}{\beta'}u_x 
\qquad
\text{in} \ (D_\Gamma\cap \{s>0\}\cap \Omega)\times (0,T), 
\end{equation}
with
\begin{equation*}
\begin{array}{l}
a_w = \dfrac{K^2}{(1-rK)^2} - \dfrac{pK}{1-rK}|\nabla u|^{p-2} u_r - \dfrac{K'}{(1-rK)^3}\dfrac{\alpha'}{\beta'}, \\
 \noalign{\vskip 1mm}
b_w = p |\nabla u|^{p-2}\nabla u - \dfrac{2K}{1-rK}N(s).
\end{array}
\end{equation*}
\end{lemma}

The following lemma contains the key inequalities that enable one 
to apply the maximum principle to the auxiliary functions $J$ and $\bar J$.

\begin{lemma}\label{parab ineq lemma}
Let $\Omega,\Gamma,\gamma,M$ be as in Notation \ref{notation1} and assume \eqref{increasing curvature},\eqref{technical condition}.
Let $J,\bar J$ be the functions defined in \eqref{J def}, \eqref{tilde J def} and define the parabolic operators
\begin{equation}\label{P operator def}
\mathcal{P} J := J_t - \Delta J - a J - b\cdot \nabla J
\end{equation}
and
\begin{equation}\label{P tilde operator def}
\bar{\mathcal{P}} \bar J := \bar{J}_t - \Delta \bar J - \bar a \bar J - \bar b \cdot \nabla \bar J,
\end{equation}
with
\begin{eqnarray*}
a &=& -\dfrac{pK}{1-rK}|\nabla u|^{p-2}u_r - \dfrac{p}{1-rK}c'dF|\nabla u|^{p-2} + \dfrac{K^2}{(1-rK)^2}\\
  & & -\dfrac{K'}{(1-rK)^3}\dfrac{\alpha'}{\beta'} - \dfrac{2}{1-rK}c'dF', \\
b &=& p|\nabla u|^{p-2}\nabla u - \dfrac{2K}{1-rK}N(s), \\
\bar{a} &=& - \alpha' \dfrac{p}{1-rK}\bar c'dF|\nabla u|^{p-2} - \dfrac{2\alpha'}{1-rK} \bar c'dF', \\
\bar{b} &=& p|\nabla u|^{p-2} \nabla u.
\end{eqnarray*}
Then we have,
\begin{equation}\label{J parab ineq 1}
\dfrac{\mathcal{P} J}{cdF} \leq \Theta (A) + \dfrac{K'}{(1-rK)^3} \dfrac{1}{\beta' cdF} \bar{J},
\quad\hbox{  in $(D_\Gamma\cap \{ s>\eta\}\cap \Omega) \times (0,T)$,}
\end{equation}
and
\begin{equation}\label{J tilde parab ineq 1}
\dfrac{\bar{\mathcal{P}} \bar{J}}{\bar cdF} \leq \Theta (\bar{A}),
\quad\hbox{  in $(D_\Gamma\cap \{ s>0\}\cap \Omega) \times (0,T)$,}
\end{equation}
with
\begin{equation}\label{Theta def}
\begin{array}{rcl}
\Theta (A)
& = & -(p-1) q \dfrac{|\nabla u|^p}{u} + \dfrac{pk}{1-rK}\dfrac{u^q |\nabla u|^{p-2}}{r^\gamma} 
+ p\dfrac{|\nabla u|^{p-1}}{r} A \\
 \noalign{\vskip 1mm}
& & - q(q-1) \dfrac{|\nabla u|^2}{u^2} + \dfrac{2q}{r}\dfrac{|\nabla u|}{u} A + \dfrac{2qk}{1-rK} \dfrac{u^{q-1}}{r^\gamma} - \dfrac{\gamma(\gamma+1)}{r^2}, 
\end{array}
\end{equation}
and $A = A(r,s) = \gamma + \dfrac{rK}{1-rK}$, $\bar{A} = \bar{A}(r,s) = \gamma + \dfrac{\tau r}{1-rK}$, 
for some $\tau=\tau(\Omega)>0$.

In addition, there exists a constant $L=L(p,\Omega,\|u_0\|_{C^1})>0$ such that, for all 
real numbers $X>0$, we have
\begin{equation}\label{Theta ineq}
\begin{array}{rcl}
\Theta ( X) & \leq & \left[k B \left( pL^{q+p-2} + 2q L^{q-1} \right) + \dfrac{q}{q-1} X^2  
  + \dfrac{\sigma}{2} X - \gamma(\gamma+1) \right] \dfrac{1}{r^2}  \\
\noalign{\vskip 1mm}
                             &     & + \left( \dfrac{p^2 X}{2\sigma} L^{p-1} - (p-1)q\right) \dfrac{|\nabla u|^p}{u},
\end{array}
\end{equation}
where $B= B(r,s) = \dfrac{r^{(q-1)(2\sigma - \frac{1}{p-1})+2}}{1-rK(s)}$.
\end{lemma}

Since the proofs of these two Lemmas require long computations, we postpone them after the proof of Theorem \ref{local result}.

\subsection{Proof of Theorem \ref{local result}}

\begin{proof}[Step 1: Preparations]
Fix any $\eta\in (0,s_0/2)$ and recall the definition of the auxiliary function $J$ given in \eqref{J def}
$$
J = \dfrac{u_s}{1-rK} + c(s) d(r) F(u)
\ = \dfrac{u_s}{1-rK} +k(s-\eta)r^{-\gamma} u^q
 \qquad \text{in}\ \omega_0\times (t_0,T),
$$ 
with $1<q<2$, $\gamma = (1-2\sigma)(q-1)>0$, where 
$\sigma\in \bigl(0,\frac{1}{2(p-1)}\bigr)$ is fixed, and $k\in (0,1)$ and $\gamma$ will be taken small (i.e., $q$ close to $1$).
Without loss of generality, by taking $r_0>0$ possibly smaller, we may assume that
$$\omega_0=M((0,r_0)\times (0,s_0))\subset \Omega,$$
where $M$ is the coordinate map defined in \eqref{flow coordinates}.

Observe that, for each $t_0< T'<T$, we have
\begin{equation}\label{grad u bounded before T'}
u\leq Cr \qquad \text{in} \ \omega_0 \times [t_0,T'],
\end{equation}
for some $C=C(T')>0$. Since $\gamma < q$, we have in particular
\begin{equation}\label{J cont before T'}
J \in C(\overline{\omega_0}\times [0,T)) \cap C^{2,1} (\omega_0\times (0,T)).
\end{equation}
Fix $t_1=\frac{t_0+T}{2}$, $s_1=\frac{3}{4}s_0$ 
and set $K_1 = \displaystyle\max_{[0,s_0]} K(s)$.
Our aim is to use the maximum principle to prove that
\begin{equation}\label{J negative}
J\leq 0 \quad \text{in} \ \omega_{1,\eta} \times (t_1,T), 
\end{equation}
where 
$$\omega_{1,\eta} :=  M((0,r_1)\times (\eta,s_1)),$$
for $r_1\in \bigl(0,\min(r_0,\textstyle\frac{1}{2K_1})\bigr)$ to be chosen below.

Note that since $1- r K\ge 1/2$ in $\omega_{1,\eta}$, inequality \eqref{J negative} implies
\begin{equation}\label{u_s negative}
u_s\leq -(1- rK) cdF \le -\textstyle\frac{k}{2} (s-\eta) r^{-(1-2\sigma)(q-1)}u^q.
\end{equation}
Hence, if \eqref{J negative} is proved, then integrating \eqref{u_s negative} over the curve 
$$\{ \gamma (\theta) + r N(\theta) ; \ \theta \in [ \eta ,s)\}$$
for $\eta<s<s_1$, $0<r< r_1$  and $t_1<t<T$, we get
$$
u\leq C (s-\eta)^{-\frac{2}{q-1}} r^{1-2\sigma}\le  C (s-\eta)^{-\frac{2}{q-1}} \delta^{1-2\sigma}(x,y) \qquad \text{in} \ \omega_{1,\eta} \times (t_1,T),
$$
for some constant $C=C(\eta)>0$.
Then, since $1-2\sigma>(p-2)/(p-1)$, 
it will follow from Lemma~\ref{nondegeneracy lemma} and symmetry
that $GBUS (u_0) \subset \gamma\bigl([-\eta, \eta]\bigr)$. Since $\eta$ is arbitrarily small, we will conclude that $GBUS (u_0)=\{(0,0)\}$.

\textit{Step 2: Parabolic inequality for $J$.}

It follows from \eqref{J parab ineq 1} and \eqref{Theta ineq} in Lemma \ref{parab ineq lemma} that, for the parabolic operator $\mathcal{P}$ defined in \eqref{P operator def}, we have
\begin{eqnarray}
\notag
&\dfrac{\mathcal{P} J}{cdF} \leq \left[ k B (p L^{q+p-2} + 2q L^{q-1}) + \dfrac{q}{q-1} A^2 + \dfrac{\sigma}{2} A - \gamma(\gamma+1) \right] \dfrac{1}{r^2} \\
\noalign{\vskip 1mm}
& \qquad\qquad +\left(\dfrac{p^2 A}{2\sigma} L^{p-1} - (p-1)q\right) \dfrac{|\nabla u|^p}{u} + \dfrac{K'}{(1-rK)^3} \dfrac{1}{\beta' cdF}\bar{J},
\quad\hbox{in $ \omega_{1,\eta} \times (t_0,T)$, } &
\label{J parab ineq 2}
\end{eqnarray}
with $L=L(p,\Omega,\|u_0\|_{C^1})>0$. 
At this point we fix $\gamma$ and~$r_1$ satisfying
\begin{equation}\label{hypGamma}
0<\gamma <\sigma\min\Bigl(\frac{1}{4}, \frac{1}{p^2L^{p-1}}\Bigr)<1
\end{equation}
and
\begin{equation}\label{hypr1}
0<r_1<\min\Bigl[r_0,1,\frac{\gamma^2}{2K_1},\frac{\gamma^2}{2\tau},\frac{3\gamma^2}{2(pL^{q+p-2} + 2qL^{q-1})}  \Bigr],
\end{equation}
where $\tau=\tau(\Omega)>0$ is given by Lemma \ref{parab ineq lemma} (some of the conditions in \eqref{hypGamma}, \eqref{hypr1} will be used only in Step 3),
and we set
$$\omega_1:=M((0,r_1)\times (0,s_1)).$$
It follows, from $r_1 < \frac{1}{2K_1}$, that 
\begin{equation}\label{1minusrK}
1- r K\ge 1/2\quad\hbox{ in $\omega_1$,}
\end{equation}
hence
\begin{eqnarray}
\label{estimAB1}
A &=& \gamma + \dfrac{rK}{1-rK}\le \gamma(\gamma+1)\quad\hbox{ in $\omega_1$,} \\ 
 \label{estimAB2}
B &=& \dfrac{r^{(q-1)(2\sigma - \frac{1}{p-1})+2}}{1-rK}\le  2r_1 
\quad\hbox{ in $\omega_1$,}
\end{eqnarray}
where we used $(q-1)\left(2\sigma - \frac{1}{p-1}\right) + 2  \geq 1$, which follows from $1<q<p$.
As a consequence of \eqref{hypGamma} and \eqref{estimAB1}, using $p>2$ and $q>1$, we first get
\begin{equation}\label{condAB3}
\dfrac{p^2 A}{2\sigma} L^{p-1} - (p-1)q \leq \dfrac{p^2 \gamma}{\sigma} L^{p-1} -  1 \leq 0
\quad\hbox{ in $\omega_1$.}
\end{equation}
Next, since $\gamma = (1-2\sigma)(q-1)$, we deduce from \eqref{hypGamma} and \eqref{estimAB1} that
$$
\begin{array}{ll}
\dfrac{q}{q-1} A^2 + \dfrac{\sigma}{2}A - \gamma(\gamma+1)
&\le \gamma(\gamma+1)\left((1-2\sigma+\gamma)(\gamma+1) +  \dfrac{\sigma}{2} -1\right) \\
\noalign{\vskip 1mm}
&= \gamma(\gamma+1)\left([\gamma+2(1-\sigma)]\gamma - \dfrac{3\sigma}{2}\right)\\
\noalign{\vskip 1mm}
&\le 3 \gamma(\gamma+1)\left(\gamma - \dfrac{\sigma}{2}\right) \le - 3\gamma^2
\quad\hbox{ in $\omega_1$.}
\end{array}
$$
In view of \eqref{hypr1}, \eqref{estimAB2}, and recalling $k\in (0,1)$, we obtain
\begin{equation}\label{condAB4}
\begin{array}{rcl}
&k B (p  L^{q+p-2} + 2q L^{q-1}) + \dfrac{q}{q-1} A^2 + \dfrac{\sigma}{2} A - \gamma(\gamma+1)  \\
\noalign{\vskip 1mm}
&\qquad\qquad\qquad\qquad\qquad\qquad \le 2r_1  (p L^{q+p-2} + 2q L^{q-1}) -3\gamma^2\le 0\quad\hbox{ in $\omega_1$.}
\end{array}
\end{equation}
It follows from \eqref{J parab ineq 2}, \eqref{condAB3}, \eqref{condAB4} that, for all $k\in (0,1)$,
\begin{equation}\label{control parab ineq 1/r^2}
 \mathcal{P} J 
 \leq \dfrac{ K'}{\beta'(1-rK)^3 } \bar{J}  \qquad \text{in} \ \omega_{1,\eta}\times (t_0,T).
\end{equation}

Moreover, in view of \eqref{increasing curvature}, \eqref{technical condition}, \eqref{grad u bounded before T'} and \eqref{1minusrK},
the coefficient $a$ in $\mathcal{P}$ satisfies
\begin{equation}\label{a bounded before T'}
\sup_{\omega_{1,\eta}\times (t_0,T')} |a| < \infty, 
\qquad \text{for any} \ T'<T.
\end{equation}

\goodbreak

\textit{Step 3: Control of $\bar{J}$.}

We claim that
under assumptions \eqref{hypGamma}, \eqref{hypr1}, there exists 
$\tilde k\in (0,1)$ such that, for all $k\in (0,\tilde k]$,
\begin{equation}\label{control u_x + cdF}
\bar{J}= u_x + \bar{c}dF
 \ = u_x + ksr^{-\gamma} u^q
 \le 0 \quad \text{in} \  \omega_1\times (t_1,T),
\end{equation}
hence
\begin{equation}\label{control PJ}
\mathcal{P} J \le 0 \quad \text{in} \  \omega_{1,\eta}\times (t_1,T).
\end{equation}
By \eqref{J tilde parab ineq 1} and \eqref{Theta ineq} in Lemma \ref{parab ineq lemma}, we have 
the following inequality for the parabolic operator $\bar{\mathcal{P}}$ defined in \eqref{P tilde operator def}:
$$
\begin{array}{rcl}
\dfrac{\bar{\mathcal{P}} \bar{J}}{\bar{c}dF} & \leq & \left[k B \left( p L^{q+p-2} + 2q L^{q-1} \right) + \dfrac{q}{q-1} \bar{A}^2  
  + \dfrac{\sigma}{2} \bar{A} - \gamma(\gamma+1) \right] \dfrac{1}{r^2}  \\
\noalign{\vskip 1mm}
                             &     & + \left( \dfrac{p^2 \bar{A}}{2\sigma}  L^{p-1} - (p-1)q\right) \dfrac{|\nabla u|^p}{u},
                             \qquad\hbox{in $\omega_1 \times (t_0,T)$},
\end{array}
$$
where
$$\bar A = \gamma + \dfrac{\tau(\Omega) r}{1-rK}
\quad\hbox{ and }\quad
B = \dfrac{r^{(q-1)(2\sigma - \frac{1}{p-1} )+2}}{1-rK}.$$
Moreover,  under  assumptions \eqref{hypGamma}, \eqref{hypr1}
(which in particular guarantee $\bar A\le \gamma(\gamma+1)$ in $\omega_1$), the argument leading to \eqref{condAB3}, \eqref{condAB4} yields:
$$
\dfrac{p^2 \bar A}{2\sigma} L^{p-1} - (p-1)q \leq 0
\quad\hbox{ in $\omega_1$,}
$$
and
$$
k B (p L^{q+p-2} + 2q L^{q-1}) + \dfrac{q}{q-1} \bar A^2 + \dfrac{\sigma}{2} \bar A - \gamma(\gamma+1) \le 0
\quad\hbox{ in $\omega_1$.}
$$
For any $k\in (0,1)$, we thus obtain
\begin{equation}\label{parab ineq tilde J}
\bar{\mathcal{P}} \bar{J}\leq 0, 
                             \qquad\hbox{in $\omega_1 \times (t_0,T)$}.
\end{equation}

By \eqref{GBUS localisation}, there exists a constant $C>0$ such that 
$$
|\nabla u|\leq C \quad \text{in} \quad \left(\omega_0\setminus M \left(\left(0,r_1/2\right)\times\left(0,\theta_1\right) \right) \right)
\times (t_0,T),
$$ 
for $\theta_1\in \left(\frac{s_0}{2},s_1\right)$.
Consequently, by parabolic estimates, $u$ can be extended to a function 
such that 
\begin{equation}\label{extensionT}
u, \nabla u\in C^{2,1}(\tilde{\mathcal{Q}}) \quad\hbox{ where }
\tilde{\mathcal{Q}} = \left(\overline{\omega_0}\setminus M \left(\left(0,\textstyle\frac{3r_1}{4}\right)\times\left(0,\theta_2\right) \right) \right)\times (t_0,T], 
\end{equation}
with $\theta_2\in (\theta_1, s_1)$.
Fix any $t_2\in (t_0,t_1)$ and $r_2\in \bigl(r_1,\min(r_0,\frac{1}{2K_1})\bigr)$. 
Since $w = u_x$ satisfies 
$$w_t - \Delta w = p|\nabla u|^{p-2}\nabla u\cdot \nabla w \quad \text{in} \ \omega_0\times(t_0,T),$$
by Hopf's Lemma, \eqref{extensionT} and \eqref{monotonicity condition},
there exist $c_1,c_2>0$ such that
\begin{eqnarray}
\label{hopf serrin tilde1}
u_x \leq -c_1 r & \text{on} \ (0,  r_2)\times\{s_1\}\times( t_2,T), \\
\label{hopf serrin tilde2}
u_x \leq -c_1 s & \text{on} \ \{r_1\} \times(0,s_1)\times ( t_2,T), \\
\label{hopf serrin tilde3}
u \leq c_2 r & \text{on} \ (0,r_1)\times \{s_1\} \times( t_2,T).
\end{eqnarray}
Moreover, in view of \eqref{monotonicity condition}, \eqref{hopf serrin tilde1}, and since $M \left( (0, r_2) \times (0,s_0)\right) \subset \Omega$,
we can apply Lemma~\ref{Serrin corner lemma 2}  to deduce the existence of $\tilde{c}_1>0$ such that
\begin{equation}\label{tilde J initial data0}
u_x (r,s,t_1) \leq -\tilde{c}_1 rs \quad \text{in} \  (0, r_1) \times (0,s_1).
\end{equation}

Now, on the lateral boundary of $\omega_1\times (t_1,T)$, we have
\begin{eqnarray}
\label{lateral bound 1 tilde}
& & \bar{J} (0,s,t) = 0 \quad \text{on} \ \{ 0\}\times (0,s_1)\times (t_1,T), \\
\label{lateral bound 2 tilde}
& & \bar{J} (r,0,t) \leq  0 \quad \text{on} \ (0, r_1)\times \{0\} \times (t_1,T), \\
\label{lateral bound 3 tilde}
& & \bar{J}( r_1,s,t) \leq  -c_1 s + k s r_1^{-\gamma} \|u_0\|_\infty^q \leq 0
\quad \text{on} \ \{ r_1 \}\times(0,s_1)\times(t_1,T), \\
\label{lateral bound 4 tilde}
& & \bar{J}(r, s_1,t) \leq  -c_1 r + k s_1  c_2^q r^{q-\gamma} \leq 0 \quad 
\text{on} \ (0, r_1)\times\{s_1\} \times (t_1,T),
\end{eqnarray}
 for any $0<k  \le\tilde k$ with $\tilde k>0$ sufficiently small,
where we used $q>\gamma+1$.  
 And at the initial time $t=t_1$,
for any $0<k \le\tilde{k}$ with possibly smaller $\tilde k>0$, inequality~\eqref{tilde J initial data0} guarantees 
\begin{equation}\label{tilde J initial data}
\bar{J} (r,s,t_1) \leq -\tilde{c}_1 rs + k s c_2^q r^{q-\gamma} \leq 0 \quad \text{in} \ (0,r_1) \times (0,s_1).
\end{equation}
Moreover, owing to \eqref{grad u bounded before T'} and \eqref{1minusrK}, 
we have 
\begin{equation}\label{a tilde bounded before T'}
\sup_{\omega_1\times (t_0,T')} \bar{a} < \infty, \qquad \text{for any} \ T'<T.
\end{equation}

Then, for any $0<k\leq \tilde{k}$, claim \eqref{control u_x + cdF} follows from \eqref{parab ineq tilde J}, 
 \eqref{lateral bound 1 tilde}--\eqref{a tilde bounded before T'} 
and the maximum principle applied to $\bar J$ in $\omega_1 \times (t_1,T)$ 
 (see Proposition 52.4 in \cite{QS}).
Note that the use of the maximum principle is justified in view of 
the regularity property \eqref{J cont before T'}, which obviously also applies for $\bar{J}$.
Finally, \eqref{control PJ} follows from \eqref{increasing curvature}, \eqref{technical condition}, \eqref{control parab ineq 1/r^2} and \eqref{control u_x + cdF}.

\textit{Step 4: Initial and boundary conditions for $J$.}

Let $w=\dfrac{u_s}{1-rK}$. 
In view of Lemma \ref{parabolic eq for u_s lemma}, \eqref{increasing curvature}, \eqref{technical condition} and \eqref{monotonicity condition}, it follows that
\begin{equation*}
w_t - \Delta w  \ - a_w w - b_w\cdot \nabla w = \dfrac{K'}{(1-rK)^3} \frac{1}{\beta'}u_x \le 0,
\qquad\hbox{  in $\omega_0\times (t_0,T)$, }
\end{equation*}
with
\begin{equation*}
\begin{array}{l}
a_w = \dfrac{K^2}{(1-rK)^2} - \dfrac{pK}{1-rK}|\nabla u|^{p-2} u_r - \dfrac{K'}{(1-rK)^3}\dfrac{\alpha'}{\beta'}, \\
 \noalign{\vskip 1mm}
b_w = p |\nabla u|^{p-2}\nabla u - \dfrac{2K}{1-rK}N(s).
\end{array}
\end{equation*}
Note in particular that $\beta'(s)$ and $1-rK$ are uniformly positive for $s\in [\eta,s_1]$ by \eqref{technical condition} and \eqref{1minusrK}.
In view of  \eqref{monotonicity condition} and \eqref{extensionT}, we may thus apply 
the strong maximum principle and Hopf's Lemma to deduce the existence of $c_3,c_4,c_5>0$ (possibly depending on $\eta$) such that
\begin{eqnarray}
\label{hopf serrin tilde B1}
u_s \leq -c_3 r & \text{on} \ (0,r_1)\times\{s_1\}\times(t_1,T), \\
u_s \leq -c_4 & \text{on} \ \{r_1\} \times(\eta,s_1)\times (t_1,T), \\
\label{serrin estimate local result}
u_s \leq -c_3 r & \text{in} \ \omega_{1,\eta}\times\{t_1\},
\end{eqnarray}
as well as
$$
u \leq c_5 r, \quad\hbox{ on $(0,r_1)\times \{s_1\} \times(t_1,T).$}
$$

Consequently, we may choose $\tilde k>0$ small enough (possibly depending on $\eta$) such that,
for any $0<k \le\tilde{k}$, on the lateral boundary of $\omega_{1,\eta}\times (t_1,T)$, we have
\begin{eqnarray}
\label{lateral bound 1}
& & J (0,s,t) = 0 \quad \text{on} \ \{ 0\}\times (\eta,s_1)\times (t_1,T), \\ 
\label{lateral bound 2}
& & J (r,\eta,t) \le  0 \quad \text{on} \ (0, r_1)\times \{\eta\} \times (t_1,T), \\
\label{lateral bound 3}
& & J (r_1,s,t) \leq  -c_4 + k s_1 r_1^{-\gamma} \|u_0\|_\infty^q \leq 0 
\quad \text{on} \ \{r_1\}\times(\eta,s_1)\times(t_1,T), \\
\label{lateral bound 4}
& & J (r, s_1,t) \leq  -c_3 r + k s_1 c_5^q r^{q-\gamma} \leq 0 \quad 
\text{on} \ (0,r_1)\times\{s_1\} \times (t_1,T),
\end{eqnarray}
where we used $q>\gamma+1$, and at the initial time $t=t_1$,
\begin{equation}\label{initial condition for J}
 J(r,s,t_1) \leq -c_3r + ks_1 c_4^q  r^{q-\gamma} \leq 0, \quad\hbox{in $\omega_{1,\eta}$}.
\end{equation}

Then \eqref{J negative} follows from  \eqref{a bounded before T'}, \eqref{control PJ},  
\eqref{lateral bound 1}--\eqref{initial condition for J} and the maximum principle
applied to $J$ in $\omega_{1,\eta} \times (t_1,T)$  
(see Proposition 52.4 in \cite{QS}).
Note that the use of the maximum principle is justified in view of \eqref{J cont before T'}.

In view of Step 1, this concludes the proof of the Theorem.
\end{proof}

\subsection{Proof of auxiliary parabolic inequalities (Lemmas \ref{parabolic eq for u_s lemma} and \ref{parab ineq lemma})}

\begin{proof}[Proof of Lemma \ref{parabolic eq for u_s lemma}]
Let $w = \dfrac{u_s}{1-rK}$, and compute, in $(D_\Gamma\cap \{s>0\} \cap \Omega  )\times (0,T)$,
\begin{eqnarray*}
w_r &=& \dfrac{u_{rs}}{1-rK} + \dfrac{K}{(1-rK)^2} u_s, \\
w_{rr} &=& \dfrac{u_{rrs}}{1-rK} + 2\dfrac{Ku_{rs}}{(1-rK)^2} + 2\dfrac{K^2}{(1-rK)^3} u_s, \\
w_s &=& \dfrac{u_{ss}}{1-rK} + \dfrac{rK'}{(1-rK)^2}u_s, \\
w_{ss} &=& \dfrac{u_{sss}}{1-rK} + 2 \dfrac{rK'}{(1-rK)^2} u_{ss} + 2 \dfrac{r^2K'^2}{(1-rK)^3} u_s + \dfrac{rK''}{(1-rK)^2} u_s.
\end{eqnarray*}
Then, using Proposition~\ref{Lap in flow coord} we get
\begin{equation}\label{laplacian w1}
\begin{array}{rcl}
\Delta w &=& w_{rr} - \dfrac{K}{1-rK} w_r + \dfrac{1}{(1-rK)^2} w_{ss} + \dfrac{rK'}{(1-rK)^3} w_s \\
 \noalign{\vskip 1mm}
         &=& \dfrac{1}{1-rK} u_{rrs} + \dfrac{K}{(1-rK)^2} u_{rs} + \dfrac{1}{(1-rK)^3} u_{sss} + \dfrac{K^2}{(1-rK)^3} u_s \\
          \noalign{\vskip 1mm}
         & & + 3 \dfrac{rK'}{(1-rK)^4} u_{ss} + \dfrac{rK''}{(1-rK)^4} u_s + 3 \dfrac{r^2 K'^2}{(1-rK)^5} u_s
\end{array}
\end{equation}
and also
\begin{equation*}
\begin{array}{rcl}
 \noalign{\vskip 1mm}
\dfrac{(\Delta u)_s}{1-rK} &=& \dfrac{1}{1-rK}u_{rrs} - \dfrac{K}{(1-rK)^2} u_{rs} + \dfrac{1}{(1-rK)^3} u_{sss} + 3 \dfrac{rK'}{(1-rK)^4} u_{ss} \\
 \noalign{\vskip 1mm}
& & + \dfrac{rK''}{(1-rK)^4} u_s + 3 \dfrac{r^2 K'^2}{(1-rK)^5} u_s - \dfrac{K'}{(1-rK)^2} u_r - \dfrac{rKK'}{(1-rK)^3} u_r \\
 \noalign{\vskip 1mm}
&=& \Delta w - \dfrac{1}{(1-rK)^2} \left( \dfrac{K'}{1-rK} u_r + 2K u_{rs} + \dfrac{K^2}{1-rK} u_s \right),
\end{array}
\end{equation*}
and replacing $u_{rs} = (1-rK)w_r - \dfrac{K}{1-rK}u_s$ and using identity \eqref{r derivative in temrs of x and s derivatives}, we obtain
\begin{equation}\label{laplacian w 2}
\begin{array}{rcl}
\Delta w &=& \dfrac{(\Delta u)_s}{1-rK} + \dfrac{2K}{1-rK}w_r - \left(\dfrac{K^2}{(1-rK)^2} - 
\dfrac{K'}{(1-rK)^3}\dfrac{\alpha'}{\beta'}\right) w \\
 \noalign{\vskip 1mm}
& & - \dfrac{K'}{(1-rK)^3} \dfrac{1}{\beta'} u_x.
\end{array}
\end{equation}
Note that the use of \eqref{r derivative in temrs of x and s derivatives} is justified since $s>0$, and then $\beta'>0$.
Then we get
\begin{equation}\label{parabolic eq w 1}
\begin{array}{rcl}
w_t - \Delta w &=& \dfrac{(|\nabla u|^p)_s}{1-rK} - \dfrac{2K}{1-rK}w_r + \left(\dfrac{K^2}{(1-rK)^2} - 
\dfrac{K'}{(1-rK)^3}\dfrac{\alpha'}{\beta'}\right) w \\
 \noalign{\vskip 1mm}
& & + \dfrac{K'}{(1-rK)^3} \dfrac{1}{\beta'} u_x.
\end{array}
\end{equation}

Now, we write
\begin{equation}\label{nabla u p s}
(|\nabla u|^p)_s = p |\nabla u|^{p-2} \nabla u\cdot (\nabla u)_s,
\end{equation}
and using \eqref{diff op in flow coord2}, we obtain
\begin{eqnarray*}
\nabla u \cdot (\nabla u)_s &=& \left( u_r N(s) + \dfrac{u_s}{1-rK}T(s) \right) \cdot \left(u_{rs}N(s) 
+ \dfrac{u_{ss}}{1-rK}T(s) + u_r(N(s))_s\right. \\
 & & \left.+ \dfrac{rK'}{(1-rK)^2} u_s T(s) + \dfrac{u_s}{1-rK}(T(s))_s\right),
\end{eqnarray*}
where $T(s)$ and $N(s)$ are defined in Notation \ref{notation1}.

We observe that
\begin{eqnarray*}
N(s)\cdot (N(s))_s = T(s)\cdot(T(s))_s=0, \\
N(s)\cdot(T(s))_s + T(s)\cdot(N(s))_s = \left( N(s)\cdot T(s)\right)_s = 0,
\end{eqnarray*} 
so we have
\begin{eqnarray*}
\nabla u\cdot(\nabla u)_s &=& \nabla u\cdot\nabla(u_s) + \dfrac{rK'}{(1-rK)^3} u_s^2 \\
&=& (1-rK)\nabla u\cdot\nabla w - w\nabla u\cdot \nabla(rK) + w\dfrac{rK'}{(1-rK)^2} u_s \\
&=& (1-rK) \nabla u\cdot\nabla w - K u_r w.
\end{eqnarray*}
Plugging this in \eqref{nabla u p s}, we obtain
\begin{equation}\label{nabla u p s 2}
\dfrac{(|\nabla u|^p)_s}{1-rK} = p |\nabla u|^{p-2} \nabla u\cdot \nabla w - \dfrac{pK}{1-rK}|\nabla u|^{p-2} u_r w,
\end{equation}
and combining this with \eqref{parabolic eq w 1}, we obtain \eqref{parabolic ineq second class}.
\end{proof}

\eject

\begin{proof}[Proof of Lemma \ref{parab ineq lemma}]

\ 

\textit{Proof of inequality \eqref{J parab ineq 1}:}
Using  Proposition~\ref{Lap in flow coord} and \eqref{laplacian w 2}, we compute,
in $(D_\Gamma\cap \{s>\eta\} \cap \Omega)\times (0,T)$,
$$
\begin{array}{rcl}
J_t &=& \dfrac{u_{ts}}{1-rK} + cdF'u_t, \\
 \noalign{\vskip 1mm}
\Delta J &=& \Delta w + cdF'\Delta u + cdF'' |\nabla u|^2 + \dfrac{2}{1-rK} c'dF'w \\
 \noalign{\vskip 1mm}
         & & + 2cd'F'u_r + cd''F - \dfrac{K}{1-rK}cd'F + \dfrac{rK'}{(1-rK)^3} c'dF \\
 \noalign{\vskip 1mm}
         &=& \dfrac{(\Delta u)_s}{1-rK} - \left( \dfrac{K^2}{(1-rK)^2} - \dfrac{K'}{(1-rK)^3}\dfrac{\alpha'}{\beta'} - \dfrac{2}{1-rK}c'dF' \right)w \\
\noalign{\vskip 1mm}
         & &  + \dfrac{2K}{1-rK} w_r + cdF'\Delta u + cdF'' |\nabla u|^2 - \dfrac{K}{1-rK}cd'F \\
 \noalign{\vskip 1mm}
         & & + cd''F + \dfrac{rK'}{(1-rK)^3} c'dF + 2cd' F' u_r - \dfrac{K'}{(1-rK)^3} \dfrac{1}{\beta'} u_x .
\end{array}
$$
Then, it follows that
\begin{equation*}
\begin{array}{rcl}
J_t - \Delta J &=& \dfrac{(|\nabla u|^p)_s}{1-rK} + \left( \dfrac{K^2}{(1-rK)^2} -\dfrac{K'}{(1-rK)^3}\dfrac{\alpha'}{\beta'} - \dfrac{2}{1-rK}c'dF' \right) w  \\
\noalign{\vskip 1mm}
               & & - \dfrac{2K}{1-rK} w_r + cdF' |\nabla u|^p - cdF''|\nabla u|^2  + \dfrac{K}{1-rK} cd'F  \\
\noalign{\vskip 1mm}
               & & - cd'' F - \dfrac{rK'}{(1-rK)^3} c'dF - 2cd'F'u_r + \dfrac{K'}{(1-rK)^3}\dfrac{1}{\beta'} u_x,
\end{array}
\end{equation*}
and  plugging \eqref{nabla u p s 2} here, we get
\begin{eqnarray*}
J_t - \Delta J &=&  \left(-\dfrac{pK}{1-rK}|\nabla u|^{p-2} u_r + \dfrac{K^2}{(1-rK)^2} - \dfrac{K'}{(1-rK)^3}\dfrac{\alpha'}{\beta'} - \dfrac{2}{1-rK}c'dF'\right)w \\
& & + p|\nabla u|^{p-2} \nabla u\cdot \nabla w - \dfrac{2K}{1-rK}w_r + cdF'|\nabla u|^p - cdF''|\nabla u|^2 - 2cd'F'u_r \\
& & + \dfrac{K}{1-rK} cd'F  - cd''F - \dfrac{rK'}{(1-rK)^3}c'dF  + \dfrac{K'}{(1-rK)^3}\dfrac{1}{\beta'} u_x.
\end{eqnarray*}

Now, we use the following identities:
\begin{eqnarray*}
w &=& J - cdF, \\
w_r &=& N(s)\cdot\nabla J - cdF' u_r - cd'F, \\
\nabla u\cdot\nabla w &=& \nabla u\cdot\nabla J - \nabla u\cdot \nabla (cdF),
\end{eqnarray*}
and
\begin{equation}\label{nabla u nabla cdF}
\begin{array}{rcl}
\nabla u \cdot \nabla (cdF) &=& u_r (cd'F + cdF'u_r) + \dfrac{u_s}{(1-Kr)^2} (c'dF + cdF'u_s) \\
&=& cdF' |\nabla u|^2 + cd'F u_r + c'dF \dfrac{u_s}{(1-rK)^2},
\end{array}
\end{equation}
hence
\begin{equation*}
\nabla u\cdot\nabla w = \nabla u\cdot\nabla J - \dfrac{1}{1-rK}c'dF J - cdF'|\nabla u|^2 - cd'Fu_r + \dfrac{1}{1-rK} cc'd^2F^2,
\end{equation*}
to obtain
\begin{eqnarray*}
J_t - \Delta J 
&=&  \left( -\dfrac{pK}{1-rK}|\nabla u|^{p-2}u_r + \dfrac{K^2}{(1-rK)^2} -\dfrac{K'}{(1-rK)^3}\dfrac{\alpha'}{\beta'} - \dfrac{2}{1-rK}c'dF' \right. \\
&&  \left. - \dfrac{p}{1-rK}c'dF|\nabla u|^{p-2} \right) J
+  \left( p|\nabla u|^{p-2}\nabla u - \dfrac{2K}{1-rK}N(s)\right) \cdot \nabla J  \\
& &- (p-1)cdF'|\nabla u|^p + \dfrac{p}{1-rK}cc'd^2F^2|\nabla u|^{p-2} \\
& & -pcd'F|\nabla u|^{p-2} u_r + \dfrac{pK}{1-rK}cdF|\nabla u|^{p-2}u_r -cdF'' |\nabla u|^2 \\
& & + \dfrac{2K}{1-rK}cdF'u_r - 2cd'F'u_r + \dfrac{3K}{1-rK} cd'F - cd''F  \\
& & + \dfrac{2}{1-rK}cc'd^2FF' - \dfrac{rK'}{(1-rK)^3}c'dF -\dfrac{K^2}{(1-rK)^2} cdF \\
& & + \dfrac{K'}{(1-rK)^3}\dfrac{1}{\beta'} \left(u_x + \alpha'cdF\right).
\end{eqnarray*}
Let $\mathcal{P}J := J_t - \Delta J -aJ - b\cdot \nabla J$,
where $a,b$ are defined in the statement of the Lemma.
Using $K,K'\geq 0$ and the definitions of $c,d,F$,  along with $\beta'>0$, $0<\alpha'\leq 1$ and $0<c\le \bar c$, 
we then have, in~$(D_\Gamma\cap \{s>\eta\}\cap \Omega)\times (0,T)$:
\begin{eqnarray*}
\dfrac{\mathcal{P}J}{cdF} &=& -(p-1)\dfrac{F'}{F}|\nabla u|^p + \dfrac{p}{1-rK}c'dF|\nabla u|^{p-2} - p\dfrac{d'}{d}|\nabla u|^{p-2} u_r \\
& & + \dfrac{pK}{1-rK}|\nabla u|^{p-2}u_r - \dfrac{F''}{F}|\nabla u|^2 + \dfrac{2K}{1-rK}\dfrac{F'}{F}u_r - 2\dfrac{d'F'}{dF}u_r \\
& & + \dfrac{3K}{1-rK}\dfrac{d'}{d} - \dfrac{d''}{d} + \dfrac{2}{1-rK}  c'dF' - \dfrac{rK'}{(1-rK)^3}\dfrac{c'}{c} - \dfrac{K^2}{(1-rK)^2} \\
& & + \dfrac{K'}{(1-rK)^3} \dfrac{1}{\beta' cdF} (u_x + \alpha' cdF) \\
&\leq & -(p-1) q \dfrac{|\nabla u|^p}{u} + \dfrac{pk}{1-rK}\dfrac{u^q |\nabla u|^{p-2}}{r^\gamma} 
+ p\dfrac{|\nabla u|^{p-1}}{r}\left(\gamma + \dfrac{rK}{1-rK}\right) \\
& & - q(q-1) \dfrac{|\nabla u|^2}{u^2} + \dfrac{2q}{r}\dfrac{|\nabla u|}{u}\left(\gamma + \dfrac{rK}{1-rK}\right) + \dfrac{2qk}{1-rK} \dfrac{u^{q-1}}{r^\gamma} - \dfrac{\gamma(\gamma+1)}{r^2} \\
& & + \dfrac{K'}{(1-rK)^3} \dfrac{1}{\beta' cdF} (u_x +  \bar c dF)
\end{eqnarray*}
that is, \eqref{J parab ineq 1}.

\eject
\textit{Proof of inequality \eqref{J tilde parab ineq 1}:}

In a similar but simpler way as in the computation for $J$ and using \eqref{psi_r psi_s}, we compute,
in $(D_\Gamma\cap \{ s>0\} \cap\Omega)\times (0,T)$,
\begin{eqnarray*}
\bar{J}_t &=& u_{tx} + \bar c dF'u_t, \\
\Delta \bar{J} &=& (\Delta u)_x +  \bar c dF'' |\nabla u|^2 +  \bar c F\left(d'' - \dfrac{K}{1-rK}d'\right) + \dfrac{rK'}{(1-rK)^3}  \bar c'dF  \\
& & + 2 \bar cd'F'u_r + \dfrac{2}{(1-rK)^2}  \bar c'dF'u_s + \bar cdF' \Delta u \\
&=& (\Delta u)_x +  \bar c dF'' |\nabla u|^2 +  \bar c F\left(d'' - \dfrac{K}{1-rK}d'\right) + \dfrac{rK'}{(1-rK)^3} \bar c'dF  \\
& & +\dfrac{2\alpha'}{1-rK}  \bar c' dF' u_x + 2 \bar c d'F' \left(u_r + \beta'\dfrac{u_y}{1-rK} \dfrac{ \bar c' d}{ \bar c d'}\right) +  \bar c dF' \Delta u.
\end{eqnarray*} 
Then we obtain
\begin{eqnarray*}
\bar{J}_t - \Delta \bar{J}  &=& (|\nabla u|^p)_x +  \bar c dF' |\nabla u|^p - \dfrac{2}{1-rK}\alpha'  \bar c'dF' u_x - \bar c dF''|\nabla u|^2 \\
 & & - 2 \bar cd'F' \left(u_r + \beta'\dfrac{u_y}{1-rK}\dfrac{ \bar c'd}{ \bar cd'}\right)  -  \dfrac{rK'}{(1-rK)^3}  \bar c'dF - \bar cF\left(d'' - \dfrac{K}{1-rK}d'\right).
\end{eqnarray*}

In view of $u_x = \bar{J} -  \bar c dF$, and using \eqref{psi_r psi_s} and \eqref{nabla u nabla cdF}, we compute
\begin{eqnarray*}
(|\nabla u|^p)_x &=& p|\nabla u|^{p-2} \nabla u\cdot \nabla u_x \\
&=& p|\nabla u|^{p-2} \nabla u\cdot \nabla \bar{J} -  p|\nabla u|^{p-2}\nabla u\cdot \nabla(\bar c dF) \\
&=& p|\nabla u|^{p-2} \nabla u\cdot \nabla \bar{J} - p \bar c dF'|\nabla u|^p - \alpha'\dfrac{p}{1-rK} \bar c'dF|\nabla u|^{p-2}u_x \\
& & -p \bar c d'F|\nabla u|^{p-2} \left(u_r + \beta' \dfrac{u_y}{1-rK}\dfrac{ \bar c'd}{\bar cd'}\right).
\end{eqnarray*}
It then follows that
\begin{equation}\label{tilde J parabolic op}
\begin{array}{rcl}
\bar{J}_t - \Delta \bar{J} &=& \bar{a} \bar{J} +\bar{b}\cdot \nabla\bar{J} - (p-1)\bar cdF'|\nabla u|^p + \alpha'\dfrac{p}{1-rK}  \bar c \bar c'd^2F^2|\nabla u|^{p-2} \\
 \noalign{\vskip 1mm}
& & - p \bar cd'F|\nabla u|^{p-2} \left( u_r + \beta'\dfrac{u_y}{1-rK}\dfrac{ \bar c'd}{\bar cd'}\right) -  \bar cdF''|\nabla u|^2 \\
 \noalign{\vskip 1mm}
& & - 2 \bar cd'F'\left( u_r + \beta'\dfrac{u_y}{1-rK}\dfrac{ \bar c'd}{\bar cd'}\right) + \dfrac{2\alpha'}{1-rK}  \bar c \bar c'd^2 FF' \\
 \noalign{\vskip 1mm}
& &  - \dfrac{rK'}{(1-rK)^3}  \bar c'dF -  \bar cF \left(d'' - \dfrac{K}{1-rK}d'\right),
\end{array}
\end{equation}
where
$$
\bar{a} = - \alpha' \dfrac{p}{1-rK} \bar c'dF|\nabla u|^{p-2} - \dfrac{2\alpha'}{1-rK}  \bar c'dF',
\quad
\bar{b} = p|\nabla u|^{p-2} \nabla u.
$$

In view of the symmetry of $\Omega$ and $\Gamma$ (assumption \eqref{increasing curvature}), we have 
$\beta'(0)= 0$.  
By the regularity of $\partial\Omega$, it follows that there exists $\tau=\tau(\Omega)>0$ such that
$$
\beta'(s) \leq \tau s, \qquad \forall s\in [0,s_0].
$$

Let $\bar{\mathcal{P}} \bar{J} := \bar{J}_t - \Delta\bar{J} -\bar{a}\bar{J} - \bar{b} \cdot \nabla \bar{J}$,
where $\bar a, \bar b$ are defined in the statement of the Lemma.
Plugging the definitions of $ \bar c,d,F$ in the expression \eqref{tilde J parabolic op}, and using the above inequality,
$K'\ge 0$ and $\alpha'\leq 1$, we obtain
\begin{eqnarray*}
\dfrac{\bar{\mathcal{P}} \bar{J}}{ \bar cdF} &=& - (p-1) \dfrac{F'}{F}|\nabla u|^p + \alpha'\dfrac{p}{1-rK}\bar c'dF|\nabla u|^{p-2} \\
& & - p\dfrac{d'}{d}|\nabla u|^{p-2} \left(u_r + \beta'\dfrac{u_y}{1-rK}\dfrac{ \bar c'd}{ \bar cd'}\right) - \dfrac{F''}{F}|\nabla u|^2 \\
& & - 2\dfrac{d'F'}{dF} \left(u_r + \beta'\dfrac{u_y}{1-rK}\dfrac{ \bar c'd}{ \bar cd'}\right) + \dfrac{2\alpha'}{1-rK}  \bar c'dF' \\
& &  - \dfrac{rK'}{(1-rK)^3} \bar c'dF  -\dfrac{d''}{d} + \dfrac{K}{1-rK}\dfrac{d'}{d} \\
&\leq & -(p-1) q \dfrac{|\nabla u|^p}{u} + \dfrac{pk}{1-rK}\dfrac{u^q |\nabla u|^{p-2}}{r^\gamma} 
+ p\dfrac{|\nabla u|^{p-1}}{r}\left(\gamma + \dfrac{\tau r}{1-rK}\right) \\
& & - q(q-1) \dfrac{|\nabla u|^2}{u^2} + \dfrac{2q}{r}\dfrac{|\nabla u|}{u}\left(\gamma + \dfrac{\tau r}{1-rK}\right) + \dfrac{2qk}{1-rK} \dfrac{u^{q-1}}{r^\gamma} - \dfrac{\gamma(\gamma+1)}{r^2}.
\end{eqnarray*}

\textit{Proof of inequality \eqref{Theta ineq}:}

Using Young's inequality we obtain, for any $X>0$,
$$
\dfrac{2q}{r} \dfrac{|\nabla u|}{u}  X \leq q(q-1) \dfrac{|\nabla u|^2}{u^2} + \dfrac{q}{q-1}\dfrac{X^2}{r^2},
$$
and
$$
p\dfrac{|\nabla u|^{p-1}}{r}  X \leq \dfrac{\sigma}{2r^2} X + \dfrac{p^2}{2\sigma}X |\nabla u|^{2p-2},
$$
hence,
\begin{equation}\label{young 1}
\dfrac{2q}{r}\dfrac{|\nabla u|}{u}  X - q(q-1) \dfrac{|\nabla u|^2}{u^2} - \dfrac{\gamma(\gamma+1)}{r^2} 
\leq \left( \dfrac{q}{q-1} X^2 - \gamma(\gamma+1)\right) \dfrac{1}{r^2},
\end{equation}
and
\begin{equation}\label{young 2}
\begin{array}{l}
-(p-1)q\dfrac{|\nabla u|^p}{u} + p \dfrac{|\nabla u|^{p-1}}{r} X \\
\noalign{\vskip 1mm}
\qquad \quad \leq \left( \dfrac{p^2 X}{2\sigma} u|\nabla u|^{p-2} - (p-1)q \right) \dfrac{|\nabla u|^p}{u} + \dfrac{\sigma}{2r^2}  X.
\end{array}
\end{equation}
Using \eqref{bernstein estimate 1}, we obtain the following estimates
\begin{equation}\label{Bernstein 1}
\dfrac{u^q |\nabla u|^{p-2}}{r^\gamma} \leq L^{q+p-2} r^{(q-1)\frac{p-2}{p-1}-\gamma}  = L^{q+p-2} r^{(q-1)(2\sigma - \frac{1}{p-1})},
\end{equation}
\begin{equation}\label{Bernstein 2}
\dfrac{u^{q-1}}{r^\gamma} \leq L^{q-1} r^{(q-1)\frac{p-2}{p-1}-\gamma}   = L^{q-1} r^{(q-1)(2\sigma - \frac{1}{p-1})},
\end{equation}
\begin{equation}\label{Bernstein 3}
u|\nabla u|^{p-2} \leq  L^{p-1},
\end{equation}
where $L=L(p,\Omega,\|u_0\|_{C^1})>0$.
Combining \eqref{young 1}-\eqref{Bernstein 3}, we obtain
$$
\begin{array}{l}
 \Theta (X) \leq \left[k\left( p L^{q+p-2} + 2q L^{q-1} \right) \dfrac{ r^{(q-1)(2\sigma - \frac{1}{p-1})+2} }{1-rK} + \dfrac{q}{q-1} X^2  
  + \dfrac{\sigma}{2}  X - \gamma(\gamma+1) \right] \dfrac{1}{r^2}  \\
\noalign{\vskip 1mm}
 \qquad \quad         + \left( \dfrac{p^2  X}{2\sigma}  L^{p-1} - (p-1)q\right) \dfrac{|\nabla u|^p}{u},
\end{array}
$$
hence \eqref{Theta ineq}.
\end{proof}

\section{Proof of Theorem \ref{second class theorem}}\label{proof of main result}

\begin{proof}
(i) We shall produce suitable initial data by means of Proposition \ref{GBU location lemma}.
Fix $\phi\in C^\infty([0,\infty))$ such that $\phi=1$ on $[0,1]$, $\phi=0$ on $[3/2,\infty)$ and $\phi'\le 0$.
Take $\rho>0$ so small that 
\begin{equation}\label{supp u_0B}
B_{\rho} (0,0)\cap \partial\Omega \subset \gamma \left(-\frac{s_0}{2},\frac{s_0}{2}\right).
\end{equation}
Let $C_1,C_2$ be given by Proposition \ref{GBU location lemma}, pick any $\eps\in (0,\rho/4)$ such that $C_1\eps^k<C_2$ and set
$$u_0(x,y)=C_2\phi\biggl(\frac{\sqrt{x^2+(y-\eps)^2})}{\eps/2}\biggr).$$
Then we immediately have \eqref{L infinity u_0} and ${\rm supp}(u_0)\subset B_\eps(0,\eps)\subset B_{\rho/2}(0,0)$.
Also, by taking $\eps>0$ possibly smaller, we get $B_\eps(0,\eps)\subset \Omega$,
hence \eqref{supp u_0} and \eqref{inf u_0}.
It thus follows from Proposition \ref{GBU location lemma} that $T(u_0)<\infty$ and, in view of \eqref{supp u_0B},
 that condition \eqref{GBUS location second class} is satisfied.

On the other hand, \eqref{u_0 symmetric second class}, and then $u_{0,x} \leq 0$ in \eqref{monotonicity condition}, are clearly satisfied. 
Moreover, by considering $\eps>0$ possibly smaller, 
the reflection properties \eqref{u_0 moving planes hyp second class} and \eqref{u_0 moving planes hyp first class} hold trivially.
In order to prove $u_{0,s} \leq 0$ in \eqref{monotonicity condition}, we can use formula \eqref{psi_r psi_s} to obtain
$$
\dfrac{u_{0,s}}{1-rK} = \dfrac{2C_2}{\eps} \phi'\left( \dfrac{\sqrt{x^2+(y-\eps)^2}}{\eps/2} \right) \dfrac{\alpha' x + \beta' (y-\eps)}{\sqrt{x^2+(y-\eps)^2}}.
$$
Then, in view of $\phi'\leq 0$ and the definition of the change of coordinates map $(x,y) = M(r,s) \ =\gamma(s)+rN(s)$,
it suffices to check that $(\gamma',\gamma+rN-\eps e_2)\ge 0$ 
for all sufficiently small $\eps, s>0$.
To do this, let us write the Taylor expansions 
$$\gamma'(s)=e_1+sR_1(s),\quad \gamma(s)=se_1+s^2R_2(s), \quad\hbox{ for all $s>0$ small},$$
where $|R_1|, |R_2|\le C_3$ for some constant $C_3>0$.
Using also $N\perp \gamma'$, it follows that
$$(\gamma',\gamma+rN-\eps e_2)=\bigl(e_1+sR_1(s),s(e_1+sR_2(s))-\eps e_2\bigr)
\ge s(1-C_3\eps-2C_3s-C_3^2s^2)\ge 0$$
for all sufficiently small $\eps, s>0$.

(ii)  The assertion will be derived as a consequence of Theorem \ref{local result}. 
For this it suffices to establish the monotonicity properties \eqref{monotonicity condition}. The proof is done in two steps.

 \textit{Step 1: Parabolic inequality.}
Consider the auxiliary function 
$$
w = \dfrac{u_s}{1-rK} \qquad \text{in} \ \mathcal{Q}_{T}:=\omega_0\times [0,T).
$$
In view of \eqref{parabolic ineq second class}, $w$ satisfies
\begin{equation}\label{parabolic ineq second class bis}
w_t - \Delta w = a_w w + b_w\cdot \nabla w + \dfrac{K'}{(1-rK)^3} \frac{1}{\beta'(s)}u_x,
\end{equation}
with
\begin{equation*}
\begin{array}{l}
a_w = \dfrac{K^2}{(1-rK)^2} - \dfrac{pK}{1-rK}|\nabla u|^{p-2} u_r - \dfrac{K'}{(1-rK)^3}\dfrac{\alpha'(s)}{\beta'(s)}, \\
 \noalign{\vskip 1mm}
b_w = p |\nabla u|^{p-2}\nabla u - \dfrac{2K}{1-rK}N(s).
\end{array}
\end{equation*}
For any $T'\in (0,T)$, we have 
$\sup_{\mathcal{Q}_{T'}} |\nabla u|<\infty$. 
Also, by hypothesis \eqref{centerout}, $1-rK$ is bounded away from $0$ in $\omega_0$. This, together with $K'\ge 0,\alpha',\beta'>0$, implies 
\begin{equation}\label{control a second class}
\sup_{\mathcal{Q}_{T'}} a_w <\infty.
\end{equation}

Since $u_0\geq 0$ in $\Omega$, by the strong maximum principle we have $u>0$ in $\Omega\times(0,T)$.
Therefore, by Hopf's lemma we get 
\begin{equation}\label{uHopf}
u_\nu < 0\quad\hbox{ on $\partial\Omega\times (0,T)$,}
\end{equation}
where $u_\nu$ is the derivative of $u$ in the outward normal direction to the boundary.
As consequence, by \eqref{x convexity}, we have
$$
u_x = \nu_x u_\nu\leq 0 \quad \text{on} \ (\partial\Omega \cap \{x>0\})\times(0,T).
$$
By the symmetry of $u_0$ and $\Omega$, we also have
$$
u_x = 0 \quad \text{on} \ [\Omega\cap\{x=0\}]\times (0,T).
$$
Now, we see that $v=u_x$ satisfies
$$
v_t - \Delta v = p|\nabla u|^{p-2}\nabla u\cdot\nabla v \quad
\text{in} \ [\Omega\cap\{x>0\}] \times (0,T).
$$
Then, after hypothesis \eqref{u_0 decreasing hypothesis second class} and the strong maximum principle, we have
\begin{equation}\label{u_x negative}
u_x < 0 \quad \text{in} \ [\Omega\cap \{x>0\} ]\times(0,T).
\end{equation} 
Since $\omega_0 \subset \Omega\cap \{x>0\}$, it follows from \eqref{parabolic ineq second class bis}, \eqref{u_x negative} and $K'\ge 0,\beta'>0$ that
\begin{equation}\label{parabolic ineq second class bis 2}
w_t - \Delta w - a_w w - b_w\cdot \nabla w \leq 0 \qquad \text{in} \ \mathcal{Q}_T.
\end{equation}

\textit{Step 2: Boundary conditions and conclusion.}
We split the boundary of $\omega_0$ in five parts:
\begin{eqnarray*}
\Gamma_1 &=& \{ (\alpha(s),\beta(s));\quad 0<s<s_0\}, \\
\Gamma_2 &=& \Omega \cap \{ x=0\},\\
\Gamma_3 &=& \partial\Omega \cap \partial \omega_0 \cap \{r>0\}, \\
\Gamma_4 &=& \Omega\cap \Lambda_{s_0}, \\
\Gamma_5 &=& \Omega \cap D_\Gamma \cap \{ y = y_0 \}.
\end{eqnarray*}
See Figures \ref{figure moving planes 1} and \ref{figure moving planes 2} for illustrations of such partitions. 
Note that $\Gamma_3$ and/or $\Gamma_5$ may be empty. 
In that case, we need not to care about them.

\begin{figure}[h]
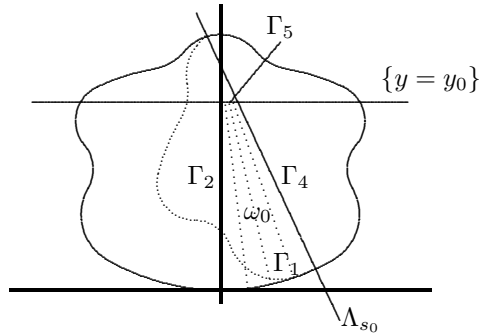

\[
\beginpicture
\setcoordinatesystem units <1cm,1cm>
\setplotarea x from -3 to 3, y from -0.2 to 3.8

\setdots <0pt>
\linethickness=1pt
\putrule from -2.8 0 to 2.8 0
\putrule from 0 -0.2 to 0 3.8

\circulararc 30 degrees from 0 0  center at 0 3
\circulararc 90 degrees from 1.50000 0.40192  center at 1.15000 1.00814
\circulararc -70 degrees from 1.75622 1.35814  center at 2.18923 1.60814
\circulararc 105 degrees from 1.80621 1.92954  center at 1.42319 2.25093
\circulararc 15 degrees from 1.63450 2.70408  center at 0.36664 -0.01484
\circulararc -40 degrees from 0.88758 2.93958  center at 1.00914 3.62895
\circulararc 50 degrees from 0.47291 3.17900  center at 0 2.78218

\circulararc -30 degrees from 0 0  center at 0 3
\circulararc -90 degrees from -1.50000 0.40192  center at -1.15000 1.00814
\circulararc 70 degrees from -1.75622 1.35814  center at -2.18923 1.60814
\circulararc -105 degrees from -1.80621 1.92954  center at -1.42319 2.25093
\circulararc -15 degrees from -1.63450 2.70408  center at -0.36664 -0.01484
\circulararc 40 degrees from -0.88758 2.93958  center at -1.00914 3.62895
\circulararc -50 degrees from -0.47291 3.17900  center at 0 2.78218

\setdots <1.5pt>
\circulararc -5 degrees from 1.2679 0.28108  center at 0 3
\circulararc -90 degrees from 1.026105 0.180925  center at 0.78669 0.83871
\circulararc 70 degrees from 0.128903 0.599296  center at -0.340944 0.428285
\circulararc -105 degrees from -0.3409444 0.928287  center at -0.340944 1.428288
\circulararc -15 degrees from -0.823909 1.557698  center at 2.073877 0.78124
\circulararc 40 degrees from -0.524207 2.2812429  center at -1.130427 2.631243
\circulararc -50 degrees from -0.44106 2.752798  center at 0.166903 2.859999
\circulararc -4 degrees from -0.306008 3.256818  center at 1.992132 1.32845

\setdots <0pt>
\setlinear 
\plot 1.5848 -0.39865  -0.31696 3.6797 /
\plot 0.13 2.5  0.8 3.3 /

\plot -2.5 2.5  2.5 2.5 / 

\setdots <2.5pt>
\plot 0.36561 0.022362  0.061386 2.5 /
\plot 0.67485 0.076890  0.11542 2.5 / 
\plot 0.97670 0.16344  0.17216 2.5 /

\put{$\Gamma_1$}[ld] at 0.7 0.35
\put{$\Gamma_2$}[cc] at -0.25 1.5
\put{$\Gamma_4$}[cc] at 1 1.5
\put{$\Lambda_{s_0}$}[lc] at 1.6 -0.4
\put{$\{y=y_0\}$}[cc] at 2.8 2.8
\put{$\Gamma_5$}[cc] at 0.8 3.5
\put{$\omega_0$}[cc] at 0.5 1

\endpicture
\]
\caption{Illustration of the partition of $\partial \omega_0$. In this case $\Gamma_3 = \emptyset$.}
\label{figure moving planes 1}
\end{figure}

\begin{figure}[h]
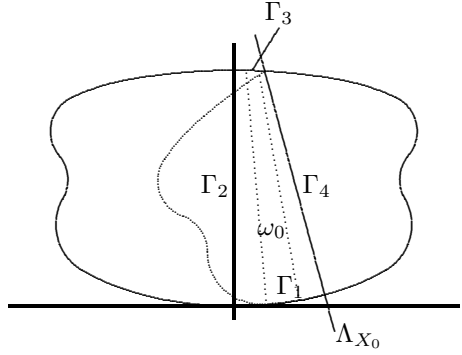

\[
\beginpicture
\setcoordinatesystem units <1cm,1cm>
\setplotarea x from -3 to 3, y from -0.2 to 3.5

\setdots <0pt>
\linethickness=1pt
\putrule from -3 0 to 3 0
\putrule from 0 -0.2 to 0 3.5

\circulararc 10 degrees from 0 0  center at 0 5
\circulararc 10 degrees from 0.86824 0.07596  center at 0.17365 4.01519
\circulararc 10 degrees from 1.54173 0.25642  center at 0.51567 3.07550
\circulararc 90 degrees from 2.01567 0.47742  center at 1.66567 1.08364
\circulararc -70 degrees from 2.27189 1.43364  center at 2.70490 1.68364
\circulararc 105 degrees from 2.32188 2.00504  center at 1.93885 2.32643
\circulararc 15 degrees from 2.15016 2.77958  center at 0.88231 0.06066
\circulararc 10 degrees from 1.40325 3.01508  center at 0.0000 -4.94275

\circulararc -10 degrees from 0 0  center at 0 5
\circulararc -10 degrees from -0.86824 0.07596  center at -0.17365 4.01519
\circulararc -10 degrees from -1.54173 0.25642  center at -0.51567 3.07550
\circulararc -90 degrees from -2.01567 0.47742  center at -1.66567 1.08364
\circulararc 70 degrees from -2.27189 1.43364  center at -2.70490 1.68364
\circulararc -105 degrees from -2.32188 2.00504  center at -1.93885 2.32643
\circulararc -15 degrees from -2.15016 2.77958  center at -0.88231 0.06066
\circulararc -10 degrees from -1.40325 3.01508  center at 0.0000 -4.94275

\setdots <1pt>
\circulararc -5 degrees from 1.208900 0.151490  center at 0.173620 4.015210
\circulararc -10 degrees from 0.868215 0.075962  center at 0.347269 3.030398
\circulararc -90 degrees from 0.347268 0.030385  center at 0.347268 0.730388
\circulararc 70 degrees from -0.352735 0.730388  center at -0.852737 0.730388
\circulararc -105 degrees from -0.681726 1.200237  center at -0.510715 1.670085
\circulararc -15 degrees from -0.920293 1.956874  center at 1.537174 0.236137
\circulararc -7 degrees from -0.391197 2.534281  center at 4.802934 -3.655845

\setdots <0pt>
\setlinear 
\plot 1.33833 -0.33147  0.28215 3.61028 /

\plot 0.25 3.15  0.6 3.7 /

\setdots <2pt>
\setlinear
\plot 0.86824 0.075961  0.32955 3.1310 /
\plot 0.43578 0.019027  0.16307 3.1361 /

\put{$\Lambda_{X_0}$}[lc] at 1.35 -0.4
\put{$\Gamma_4$}[cc] at 1.1 1.6
\put{$\Gamma_2$}[cc] at -0.25 1.6
\put{$\omega_0$}[cc] at 0.5 1
\put{$\Gamma_1$}[cc] at 0.75 0.25
\put{$\Gamma_3$}[cc] at 0.6 3.9

\endpicture
\]
\caption{Illustration of the partition of $\partial\omega_0$ when $y_0 = \infty$. In this case $\Gamma_5 = \emptyset$.}
\label{figure moving planes 2}
\end{figure}

Since $u=0$ on $\partial\Omega$, we have
\begin{equation}\label{usGamma1}
u_s= 0\quad\hbox{ on $\Gamma_1\times [0,T)$.}
\end{equation}
By the symmetry of the domain, and using \eqref{psi_r psi_s}, we have $u_s=(1-rK)u_x$ on $\Gamma_2$, and by~\eqref{u_0 symmetric second class}, 
we deduce 
\begin{equation}\label{usGamma2}
u_s=0\quad\hbox{ on $\Gamma_2\times [0,T)$.}
\end{equation}
On the other hand, as consequence of \eqref{uHopf}, \eqref{x convexity} and \eqref{thin domain condition 2 1}, we have
$$
u_x = \nu_x u_\nu \leq 0, \qquad u_y = \nu_y u_\nu \leq 0 \qquad\hbox{ on $\Gamma_3\times [0,T)$.}
$$
Now, we recall from \eqref{psi_r psi_s}
\begin{equation}\label{usGamma3aux}
u_s = (1-rK) (\alpha'(s) u_x + \beta'(s) u_y)
\end{equation}
Then it follows from \eqref{technical condition} that
\begin{equation}\label{usGamma3}
u_s \leq 0 \qquad \text{on} \ \Gamma_3\times [0,T).
\end{equation}

Next, we shall prove by a moving planes argument that
\begin{equation}\label{usGamma4}
u_s\leq 0\quad\hbox{ on $\Gamma_4\times [0,T)$. }
\end{equation}
We define in $\Omega_{s_0}\times (0,T)$ the functions
$$u_1(x,y,t) = u(x,y,t), \qquad u_2(x,y,t) = u(\mathcal{T}_{s_0}(x,y),t), $$
where $\Omega_{s_0}$ and $\mathcal{T}_{s_0}$ are defined in Notation \ref{notation2}.
We note that $u_2$ is well defined since by condition \eqref{thin domain condition 2 2}, $\mathcal{T}_{s_0} ( x, y)\in \Omega$, for all $(x,y)\in \Omega_{s_0}$.
Both functions $u_1, u_2$ satisfy the equation
$$
u_{i,t} - \Delta u_i = |\nabla u_i|^p \qquad \text{in} \ \Omega_{s_0}\times (0,T),
$$
for $i=1,2$. 
By condition \eqref{u_0 moving planes hyp second class}, we have
$$
u_1(x,y,0) \leq u_2(x,y,0) \qquad \text{in} \ \Omega_{s_0}.
$$
The boundary of $\Omega_{s_0}$ is composed of two parts:
$$
\Gamma_1^{s_0} := \partial\Omega\cap\partial\Omega_{s_0},\qquad \Gamma_2^{s_0} := \Lambda_{s_0}\cap \Omega.
$$
On $\Gamma_1^{s_0}$ we have $u_1(x,y,t) = 0$ and $u_2(x,y,t)\geq 0$ since $u\geq 0$ in $\Omega\times (0,T)$. 
On $\Gamma_2^{s_0}$ we have $u_1(x,y,t) = u_2(x,y,t)$, since $\mathcal{T}_{s_0} (x,y) = (x,y)$, for all $(x,y)\in \Lambda_{s_0}$.
So we conclude that $u_1(x,y,t)\leq u_2(x,y,t)$ on $\partial\Omega_{s_0}\times [0,T)$.
As a consequence of the comparison principle, we get $u_1\leq u_2$ in $\Omega_{s_0}\times [0,T)$.
Letting $(x,y)$ go to $\Lambda_{s_0}$ in the normal direction to $\Lambda_{s_0}$, we deduce \eqref{usGamma4}.

In order to show that 
\begin{equation}\label{usGamma5}
u_s\leq 0\quad\hbox{ on $\Gamma_5\times [0,T)$,}
\end{equation}
we observe that, as a consequence of \eqref{Omega plus}, \eqref{u_0 moving planes hyp first class}
and of a similar moving planes argument as in the case of $\Gamma_4$, we have
\begin{equation}\label{usGamma5aux}
u_y \leq 0\quad\hbox{ in $(\Omega\cap \{y=y_0\})\times [0,T)$.}  
\end{equation}
Property \eqref{usGamma5} then follows from \eqref{u_x negative},  \eqref{usGamma3aux},  \eqref{usGamma5aux}
and \eqref{technical condition}.

Then, since $\partial\omega_0 = \Gamma_1\cup\Gamma_2\cup\Gamma_3\cup\Gamma_4\cup \Gamma_5$, it follows
from \eqref{usGamma1}, \eqref{usGamma2}, \eqref{usGamma3}-\eqref{usGamma5} that
\begin{equation}\label{boundary condition conclusion}
 u_s \leq 0\qquad \text{on} \ \partial \omega_0 \times (0,T).
\end{equation}
In view of \eqref{control a second class}, \eqref{parabolic ineq second class bis 2}, \eqref{boundary condition conclusion}, \eqref{u_0 decreasing hypothesis second class}, 
it follows from the strong maximum principle that
$$
u_s < 0 \qquad \text{in} \ \omega_0\times (0,T).
$$
This, together with \eqref{u_x negative} and \eqref{GBUS location second class}, allows us to apply Theorem \ref{local result}, and the conclusion follows.
\end{proof}

\section{Proof of Theorems \ref{ellipses theorem} and \ref{almost flat theorem}}\label{reduction section}

Here we give the proofs of Theorems \ref{ellipses theorem} and \ref{almost flat theorem} as consequence of Theorem \ref{second class theorem}.
We shall verify that the hypotheses of Theorem \ref{second class theorem} hold 
for ellipses and for the domains satisfying the assumptions of Theorem \ref{almost flat theorem}. 

\begin{proof}[Proof of Theorem \ref{ellipses theorem}]
We only give the proof for ellipses with positive eccentricity. For disks, see \cite{Li-Souplet}.
Without loss of generality, we assume that the minor axis of the ellipse is on the half-line $\{x=0; \ y\geq 0\}$
and that the lower co-vertex is at the origin.
Then, assumption \eqref{x convexity} holds.
If we consider $\Gamma$ a connected boundary piece containing the origin and symmetric with respect to $x=0$, we can use Notation \ref{notation1}. 

Now, take $y_0>0$ such that the major axis of the ellipse is on the line $y=y_0$.
Hypothesis \eqref{Omega plus} is then satisfied.
Moreover, in view of the position of the ellipse,  it is well known that the center of curvature at any point of
$\partial\Omega \cap \{y<y_0\}$
lies in the half-plane $\{y>y_0\}$.
Considering $s_0>0$  small enough so that $\Gamma\subset\{y<y_0\}$,
it follows that conditions \eqref{increasing curvature}, \eqref{technical condition}, \eqref{centerout} and \eqref{thin domain condition 2 1}
are satisfied.

Now let us verify that \eqref{thin domain condition 2 2} also holds for this choice of $\Gamma$.
Here, we recall the definitions of $H_{s_0}$ and $\Lambda_{s_0}$ in Notation \ref{notation2}.
We shall prove that the symmetric of $\partial\Omega\cap H_{s_0}$ with respect to $\Lambda_{s_0}$ lies in $\Omega$,
which guarantees \eqref{thin domain condition 2 2} by convexity.

Let $\partial\Omega$ be the original ellipse and $\mathcal{T}_{s_0}(\partial\Omega)$ its symmetric with respect to the line $\Lambda_{s_0}$.
We observe that the two ellipses intersect in at least two points, which are the two intersection points of $\partial\Omega$ with $\Lambda_{s_0}$. We also know that any two ellipses intersect in at most four points, counting the multiplicity.
Since $\Lambda_{s_0}$ is normal to $\partial\Omega$ at $\gamma(s_0)$, the two ellipses $\partial\Omega$ and $\mathcal{T}_{s_0}(\partial\Omega)$ are tangent to each other at that point, 
which is then an intersection point of multiplicity at least 2. 

Therefore, there can be at most one other 
intersection point between the two ellipses. 
By convexity, it cannot be on the segment $\Lambda_{s_0}\cap \Omega$, and by symmetry with respect to the line $\Lambda_{s_0}$, 
if there is an intersection point on one side of $\Lambda_{s_0}$, there must be another one on the other side. 
Hence, the two ellipses only intersect in two points.

Finally, since the curvature of $\partial\Omega$ increases from the origin up to the right vertex, 
near $\gamma(s_0)$, the symmetric of $\partial\Omega\cap H_{s_0}$ lies in $\Omega$.
As we have seen, it does not intersect again the boundary of $\Omega$ until the other intersection of $\Lambda_{s_0}$ with $\partial\Omega$. 
Therefore, we conclude that $\mathcal{T}_{s_0} (\partial\Omega\cap H_{s_0}) \subset\Omega$.
Hence, $\Omega$ satisfies all the hypothesis of Theorem \ref{second class theorem}, and the conclusion follows.
\end{proof}

\begin{proof}[Proof of Theorem \ref{almost flat theorem}]
We give the proof for the case when $\Omega$ is not locally flat at the origin. 
That is, in assumption \eqref{Omega tangent to y>0} $\partial\Omega$ only touches $y=0$ at the origin.
For locally flat domains, see \cite{Li-Souplet}.

As in the proof of Theorem~\ref{ellipses theorem}, we shall verify that all the hypotheses of Theorem~\ref{second class theorem} hold.
In view of assumptions \eqref{x symmetry and convexity pre} and \eqref{Omega tangent to y>0}, and considering a suitable boundary piece, we can use Notation~\ref{notation1}, and hypothesis \eqref{x convexity} is satisfied.
By taking a smaller $\Gamma$ if necessary, hypotheses \eqref{increasing curvature}, \eqref{technical condition} and \eqref{thin domain condition 2 2} are guaranteed by assumptions \eqref{small curvature at the origin pre} and \eqref{symmetry K(0) pre}.

The assumption $\overline{\Omega}\subset \{y<R (0)\}$ in \eqref{small curvature at the origin pre}, implies that
the center of curvature of the boundary at the origin is at positive distance of $\Omega$ (possibly at infinity).
Since the curvature is a continuous function
due to the regularity of the boundary, 
considering a smaller $\Gamma$ if necessary, the evolute of 
$\Gamma$ is also at positive distance of $\Omega$.
Therefore, hypothesis \eqref{centerout} is satisfied with $y_0=+\infty$, and then \eqref{Omega plus} is trivial.

As for hypothesis  \eqref{thin domain condition 2 1}, we note that  in view of \eqref{x symmetry and convexity pre} and since $\Omega$ is smooth and 
connected, $\Omega\cap \{x=\eta\}$ is a segment for all $\eta>0$ small.
Therefore, \eqref{thin domain condition 2 1} holds by considering a possibly smaller $\Gamma$. 
\end{proof}

{\bf Acknowledgements.} 
The author was partially supported by Sorbonne Universit\'e, Laboratoire Jaques-Louis Lions (LJLL)
Paris, France.


\begin{thebibliography}{99}

\bibitem{A}
{\sc N. Alaa},
{Weak solutions of quasilinear parabolic equations with measures as initial data}.
{\it Annales Math\'ematiques Blaise Pascal} 3, no. 2 (1996): 1-15.

\bibitem{ABG}
{\sc N.D. Alikakos, P.W. Bates, C.P. Grant},
{Blow up for a diffusion-advection equation}.
{\it Proceeding of the Royal Society of Edinburgh},
Section A 113, no. 3-4 (1989): 181-90.

\bibitem{ABA}
{\sc L. Amour, M. Ben-Artzi},
{Global existence and decay for viscous Hamilton-Jacobi equations}.
{\it Nonlinear Analysis} 31 (1998): 621-28.

\bibitem{ARBS}
{\sc J.M. Arrieta, A. Rodriguez-Bernal, Ph. Souplet},
{Boundedness of global solutions for nonlinear parabolic equations involving gradient blow-up phenomena.} 
{\it Annali della Scuola Normale Superiore di Pisa}, Classe di Scienze (5)3, no. 1 (2004): 1-15.

\bibitem{AS17}
{\sc A. Attouchi, Ph. Souplet},
Single point gradient blow-up on the boundary for a Hamilton-Jacobi equation with $p$-Laplacian diffusion,
{\it Trans. Amer. Math. Soc.} 369 (2017), 935-974.

\bibitem{BDL}
{\sc G. Barles, F. Da Lio},
{On the generalized Dirichlet problem for viscous Hamilton-Jacobi equations}.
{\it Journal de Math\'ematiques Pures et Appliqu\'ees} 83 (2004): 53-75.

\bibitem{BKL}
{\sc S. Benachour, G. Karch, Ph. Lauren\c{c}ot},
{Asymptotic profiles of solutions to viscous Hamilton-Jacobi equations}.
{\it Journal de Math\'ematiques Pures et Appliqu\'ees} 83 (2004): 1275-308.

\bibitem{BL}
{\sc S. Benachour, Ph. Lauren\c{c}ot},
{Global solutions to viscous Hamilton-Jacobi equations  with irregular data}.
{\it Communications in Partial Differential Equations} 24 (1999): 1999-2021.

\bibitem{BASW}
{\sc M. Ben-Artzi, Ph. Souplet, F.B. Weissler},
{The local theory for viscous Hamilton-Jacobi equations in Lebesgue spaces}.
{\it Journal de Math\'ematiques Pures et Appliqu\'ees} (9) 81, no. 4 (2002): 343-78.

\bibitem{CG}
{\sc G.R. Conner, C.P. Grant},
{Asymptotics of blowup for a convection-diffusion equation with conservation,}
{\it Differential Integral Equations} 9 (1996), 719-728.

\bibitem{FL}
{\sc M. Fila, G.M. Lieberman},
{Derivative blow-up and beyond for quasilinear parabolic equations}.
{\it Differential and  Integral Equations} 7, no. 3-4 (1994): 811-21.

\bibitem{Friedman} 
{\sc A.  Friedman},
{Partial Differential equations of parabolic type},
Englewood cliffs, NJ: Prentice-Hall,1964.

\bibitem{FM} 
{\sc A. Friedman, B. McLeod},
{Blow-up of positive solutions of semilinear heat equations},
{\it Indiana Univ. Math.} J.  34 (1985), 425-447.

\bibitem{G}
{\sc B.H. Gilding},
{The Cauchy problem for $u_t = \Delta u + |\nabla u|^q$, large-time behaviour}.
{\it Journal de Math\'ematiques Pures et Appliqu\'ees} (9) 84, no. 6 (2005): 753-85.

\bibitem{GGK}
{\sc B.H. Gilding, M. Guedda, R. Kersner},
{The Cauchy problem for $u_t = \Delta u + |\nabla u|^q$}.
{\it Journal of Mathematical Analysis and Applications} 284 no. 2 (2003): 733-55.

\bibitem{GH}
{\sc J.-S. Guo, B. Hu}, 
{Blowup rate estimate for the heat equation with a nonlinear gradient source term},
{\it Discrete Contin. Dyn. Syst.} 20 (2008), 927-937.

\bibitem{HH-Z}
{\sc T. Halpin-Healy, Y.C. Zhang},
{Kinetic roughening phenomena, stochastic growth, directed polymers and all that},
{\it Aspects of multidisciplinary statistical mechanics.}  Phys. Rev. 254, 215-414 (1995).

\bibitem{HM}
{\sc M. Hesaaraki, A. Moameni},
{Blow-up positive solutions for a family of nonlinear parabolic equations in general domain in $\mathbb{R}^N$}.
{\it Michigan Mathematical Journal} 52, no. 2 (2004): 375-89.

\bibitem{KPZ}
{\sc M. Kardar, G. Parisi, Y.C. Zhang},
{Dynamic scaling of growing interfaces.} Phys. Rev. Lett. 56(9), 889-892(1986).

\bibitem{KS}
{\sc J. Krug, H. Spohn},
{Universality classes for deterministic surface growth.} Phys. Rev. A 38, 4271-4283 (1988).


\bibitem{LS}
{\sc Ph. Lauren\c{c}ot, Ph. Souplet},
{On the growth of mass for a viscous Hamilton-Jacobi equation}.
{\it Journal d'Analyse Math\'ematique} 89 (2003): 367-83.

\bibitem{Li-Souplet}
{\sc Y.-X. Li, Ph. Souplet},
{Single-Point Gradient Blow-up on the Boundary for Diffusive Hamilton-Jacobi Equations in Planar Domains},
{\it Comm. Math. Phys.} 293 (2009), 499-517.

\bibitem{Lions}
{\sc P.-L. Lions},
{Generalized solutions of Hamilton-Jacobi equations},
{\it Research Notes in Mathematics},
69. Boston, MA: Pitman (Advanced Publishing Program), 1982.

\bibitem{PS}
{\sc A. Porretta, Ph. Souplet},
{The profile of boundary gradient blow-up for the diffusive Hamilton-Jacobi equation},
{\it International Math. Research Notices} 17 (2017), 5260-5301.

\bibitem{PS2a}
{\sc A. Porretta, Ph. Souplet},
{Analysis of the loss of boundary conditions for the diffusive Hamilton-Jacobi equation},
{\it Ann. Inst. H. Poincar\'e Anal. Non Lin\'eaire} 34 (2017), 1913-1923.

\bibitem{PS2}
{\sc A. Porretta, Ph. Souplet},
Blow-up and regularization rates, loss and recovery of boundary conditions for the superquadratic viscous Hamilton-Jacobi equation,
 {\it J. Math. Pures Appl.,} to appear
(Preprint arXiv 1811.01612).


\bibitem{PZ} 
{\sc A. Porretta, E. Zuazua},
Null controllability of viscous Hamilton-Jacobi equations,
{\it Ann. Inst. H. Poincar\'e Anal. Non Lin\'eaire} 29 (2012), 301-333.

\bibitem{QR18} 
{\sc A. Quaas, A. Rodr\'\i guez},
Loss of boundary conditions for fully nonlinear parabolic equations with superquadratic gradient terms,
{\it J. Differential Equations} 264 (2018), 2897-2935.

              
\bibitem{QS}
{\sc  P. Quittner, Ph. Souplet},
{Superlinear parabolic problems. Blow-up, global existence and steady states}.
{\it Birkh\"auser Advanced Texts}, 2007.

\bibitem{S1}
{\sc Ph. Souplet},
{Gradient blow-up for multidimensional nonlinear parabolic equations with general boundary conditions}.
{\it Differential and integral equations} 15, no. 2 (2002): 237-56.

\bibitem{S2}
{\sc Ph. Souplet},
{A remark on the large-time behavior of solutions of viscous Hamilton-Jacobi equations}.
{\it Acta Mathematica Universitatis Comenianae (N.S.)} 76 (2007): 11-13.

\bibitem{SV}
{\sc Ph. Souplet, J.L. V\'azquez},
{Stabilization towards a singular steady state with gradient blow-up for a diffusion-convection problem}.
{\it Discrete and Continuous Dynamical Systems} 14, no. 1 (2006): 221-34. 

\bibitem{Souplet-Zhang}
{\sc Ph. Souplet, Q.S. Zhang},
{Global solutions of inhomogeneous Hamilton-Jacobi equations},
{\it J. Anal. Math.} 99 (2006), 355-396.

\bibitem{ZL}
{\sc Z.-C. Zhang, Z. Li},
{A note on gradient blowup rate of the inhomogeneous Hamilton-Jacobi equations},
{\it Acta Math. Sci. Ser. B} (Engl. Ed.) 33 (2013), 678-686.

\end{thebibliography}
\end{document}